\documentclass[11pt,fleqn]{amsart} 

\usepackage[usenames,dvipsnames,condensed]{xcolor}

\usepackage{amsthm}
\usepackage{amsfonts}
\usepackage[english]{babel}
\usepackage[usenames]{xcolor}
\usepackage{graphicx}
\usepackage{soul}
\usepackage{stfloats}
\usepackage{morefloats}
\usepackage{cite}
\usepackage{lscape}
\usepackage{epstopdf}
\usepackage{braket}
\usepackage[lite]{amsrefs}
\usepackage{mathbbol}
\usepackage{tikz-cd}
\usepackage{shuffle}

\usepackage{amsmath,amstext,amsopn,amsfonts,eucal,amssymb}
\usepackage{graphicx,wrapfig,url}

\newtheorem{theorem}{Theorem}[section]
\newtheorem{corollary}[theorem]{Corollary}
\newtheorem{lemma}[theorem]{Lemma}
\newtheorem{proposition}[theorem]{Proposition}
\newtheorem{definition}[theorem]{Definition}

\newtheorem{example}[theorem]{Example}

\newtheorem{remark}[theorem]{Remark}

\begin{document}
	\title[Cohomological deformations]{Deformations of Yang-Baxter operators via $n$-Lie algebra cohomology}
	
	\author[Elhamdadi]{Mohamed Elhamdadi} 
	\address{University of South Florida, Tampa, FL, USA}
	\email{emohamed@usf.edu}
	
	\author[Zappala]{Emanuele Zappala} 
	\address{Yale University, New Haven, CT, USA} 
	\email{emanuele.zappala@yale.edu \\ zae@usf.edu}

	\maketitle
	
   \begin{abstract}
   	We introduce a cohomology theory of $n$-ary self-distributive objects in the tensor category of vector spaces that classifies their infinitesimal deformations. For $n$-ary self-distributive objects obtained from $n$-Lie algebras we show that ($n$-ary) Lie cohomology naturally injects in the self-distributive cohomology and we prove, under mild additional assumptions, that the map is an isomorphism of second cohomology groups. This shows that the self-distributive deformations are completely classified by the deformations of the Lie bracket. This theory has important applications in the study of Yang-Baxter operators as the self-distributive deformations determine nontrivial deformations of the Yang-Baxter operators derived from $n$-ary self-distributive structures. In particular, we show that there is a homomorphism from the second self-distributive cohomology to the second cohomology of the associated Yang-Baxter operator. Moreover, we prove that when the self-distributive structure is induced by a Lie algebra with trivial center, we get a monomorphism. We construct a deformation theory based on simultaneous deformations, where both the coalgebra and self-distributive structures are deformed simultaneously. We show that when the Lie algebra has trivial cohomology (e.g. for semi-simple Lie algebras) the simultaneous deformations might still be nontrivial, producing corresponding Yang-Baxter operator deformations.
   	We provide examples and computations in low dimensions, and we completely characterize $2$-cocycles for the self-distributive objects obtained from all the nontrivial real Lie algebras of dimension $3$, i.e. the Bianchi I-IX, and all the nontrivial complex Lie algebras of dimension $3$. 
   \end{abstract}
	
	\section{Introduction}
	
	The Yang-Baxter Equation (YBE) was first introduced in the 1970's in statistical mechanics to impose the conservation of momenta in scattering processes \cite{Jim}. It has ever since outgrown its original field of applications by becoming a central player in geometric topology, more specifically in knot theory. In fact, YBE is an algebraic version of Reidemeister move III, which is one of the necessary and sufficient conditions for two diagrams to represent isotopic knots/links \cite{Oht,Tur}. The $n$-simplex equation, originally introduced by Zamolodchikov, is a generalization of the YBE that has been applied (for $n=3$) to the theory of knotted surfaces embedded in $4$-space \cite{nsimplex,KTS}.
	
	A well known class of algebraic structures that produce solutions to the YBE are quandles. These are sets endowed with an operation that corresponds to Reidemeister moves. In particular, they satisfy the self-distributive (SD) condition, which is a set-theoretic version of Reidemeister move III. Linearization then gives YB operators. The cohomology of quandles was introduced in 
	 \cite{CJKLS} to construct certain partition functions that are invariants of links in $3$-space and knotted surfaces in $4$-space.
	
	More recently, the study of homology theories of Yang-Baxter (YB) operators has led to new knot invariants (\cite{CES}) that generalize the {\it quandle cocycle} invariants  given in \cite{CJKLS}. Related homology theories for YB operators have appeared in \cite{PW,PW_Hom,LV}, and a skein-theoretic approach to both homology and cohomology of YB operators was given in \cite{ESZ_YB}. The systematic study of solutions of the YBE and the (co)homology theories of YB operators is still a highly nontrivial problem with applications that span from knot theory to quantum gravity. 
	
	$n$-Lie algebras are $n$-ary counterparts to Lie algebras, where the notion of Lie bracket and Jacobi identity are replaced by $n$-ary brackets and the Filippov identity. Interestingly, $n$-Lie algebras were introduced in theoretical physics \cite{Nambu} by Nambu, with the purpose of generalizing Hamiltonian dynamics. They were then rediscovered by Filippov in the 1980's \cite{Fil}, who proved that the structure 
	considered by Nambu satisfied a certain generalization of the Jacobi identity, that was later dubbed Filippov identity. A cohomology theory for $n$-Lie algebras was then constructed by Takhtajan \cite{Tak} based on the Chevalley-Eilenberg complex. This cohomology theory in particular characterizes deformations of $n$-Lie algebras providing a parallel with the deformation theory of (binary) Lie algebras due to Nijenhuis-Richardson \cite{NR}, which was itself based on the analogous deformation theory for associative algebras of 
Gerstenhaber \cite{Gerstenhaber}.

	In this article we consider $n$-ary quandle objects in the category of vector spaces. These are analogues of quandles where the defining axioms are translated from the symmetric monoidal category of sets with the cartesian product, to that of vector spaces with tensor product. More specifically, we are interested in the property of self-distributivity (in the categorical setting) which directly corresponds to Reidemeister move III. We call these structures {\it self-distributive objects}, or SD objects for short. When we want to mention explicitly the arity of the SD object, we will write $n$-SD object. We define a cohomology theory for SD objects in vector spaces that generalizes the cohomology theory of \cite{CES}, and we show that this cohomology characterizes the deformation theory of SD objects, similarly to the aforementioned cases of associative and Lie algebras. We give a general construction that derives $n$-SD objects from $n$-Lie algebras. We relate the deformation theory of $n$-Lie algebras to the deformation theory of the corresponding $n$-SD object and show that under mild hypotheses the second cohomology groups are isomorphic. Then we consider simultaneous deformations of SD objects, where both coalgebra and SD morphism are deformed together. This provides a more general class of SD deformations. We use these results to obtain deformations of YB operators arising from the $n$-SD objects and we obtain a relation with the second YB cohomology group. Along the way, we consider detailed computations in low dimensions to show that the theory is nontrivial and give tables of computatoins of the SD cohomology groups of the Lie algebras appearing in the Bianchi classification.
	
	The starting point of this article are results found in \cite{CCES}, where it was given the construction of SD objects from Lie algebras in the binary case. We also point out that in \cite{CCES} the cohomology defined becomes substantially equivalent to the binary case of the simultaneous deformations proposed here. We observe that other than considering the extended case to $n$-ary structures, we also focus on the relations between the deformation theory (i.e. cohomology) of the $n$-Lie structures and their corresponding  $n$-ary quandles. In addition, we consider in detail the application of this theory to deformations of YB operators and their second cohomology groups.
	
	The article is structured as follows. In Section~\ref{sec:Lie_deform} we recall some basic definitions regarding the deformation of Lie algebras, and cochain complexes of $n$-ary Lie algebras. Section~\ref{Prelim} gives a brief review of self-distributive objects in symmetric monoidal category needed for the rest of the article.  In Section~\ref{sec:TSD_coh} the definition of cohomology of $n$-ary quandles is given, and some properties of the first and second cohomology groups are proved. In Section~\ref{sec:TSDdeform} we prove that this newly defined cohomology characterizes deformations of $n$-ary quandle objects. Section~\ref{sec:Lie_SD} is devoted to studying the relation between the cohomology of Lie algebras and the cohomology of the corresponding SD ($n$-quandle) object. Section~\ref{sec:Coalgebradeformation} is a 
	digression on deformations of coalgebra structures, which is going to be used in the subsequent Section~\ref{sec:sim_deform}, where simultaneous deformations of $n$-quandle objects are introduced. In Section~\ref{sec:examples} we provide detailed computations and examples in low dimensions. Here, tables containing the cohomology groups for low dimensional Lie algebras are also provided. In Section~\ref{sec:YB_deform} we apply the theory introduced and developed in this article to Yang-Baxter operator deformations, and relate the results to the second cohomology group of Yang-Baxter operators. Several detailed proofs which consist of cumbersome computations are deferred to the Appendix. 
 	
	\section{Infinitesimal deformations of $n$-ary Lie algebras}\label{sec:Lie_deform}
	
    In this section we discuss deformation theory of Lie algebras by means of cohomology theory. We recall some known traditional facts for the case of binary Lie algebras, and then pass on considering $n$-Lie algebras, in the sense of Filippov. 
	
	\subsection{Binary Lie algebras}
	
	The material of this subsection is standard, and can be found in \cite{GS,GS2,NR} for example. Let $\mathfrak g$ be a (binary) Lie algebra over the field (or ring) $\mathbb k$ with bracket $[\bullet, \bullet]: \mathfrak g\otimes \mathfrak g \longrightarrow \mathfrak g$. We extend the coefficients of $\mathfrak g$ to $\mathbb k[[\hbar]]$, the power series ring on the variable $\hbar$ 
	with coefficients on $\mathbb k$. Let $\mathfrak g'$ denote the resulting module, namely $\mathfrak g' = \mathbb k[[\hbar]]\otimes_{\mathbb k} \mathfrak g$. We define a family of maps $[\bullet, \bullet]' : \mathfrak g'\otimes \mathfrak g' \longrightarrow \mathfrak g'$ as a formal power series $\sum_n [\bullet, \bullet]_n\hbar^n$, where we set $[\bullet,\bullet]_0 = [\bullet,\bullet]$. If the brackets $[\bullet, \bullet]_n$ are such that $[\bullet, \bullet]'$ satisfies the Jacobi identity, then we say that this is a Lie deformation of the bracket $[\bullet, \bullet]$, since by definition we have that on $\mathfrak g$ the two brackets coincide. This is equivalent to saying that $[\bullet,\bullet]'$ satisfies the Jacobi identity on elements $x,y,z\in \mathfrak g$ seen as elements of $\mathfrak g'$ with coefficients in $\mathbb k \hookrightarrow \mathbb k[[\hbar]]$. Here we are interested mainly in infinitesimal deformations. These are those deformations of $[\bullet,\bullet]$ up to the maximal ideal $(\hbar^2)$ 
	of $\mathbb k[[\hbar]]$. More specifically, we consider the quotient ring 
	$R := \mathbb k[[\hbar]]/(\hbar^2)$ and let  $\tilde{\mathfrak g} := R\otimes_{\mathbb k} \mathbb k[[\hbar]]$. Then a deformation on $\tilde{\mathfrak g}$ needs to satisfy the Jacobi identity on any choice of $x,y,z\in \mathfrak g$. We set, for short, $\phi(x,y) = [x,y]_1$, the bracket corresponding to the first power of $\hbar$. Then the Jacobi identity on $x,y,z$ gives 
	\begin{eqnarray*}
	[[x,y]',z]' + [[y,z]',x]' + [[z,x]',y]' &=& [[x,y]+\phi(x,y),z]' + [[y,z]+\phi(y,z),x]' + [[z,x]+\phi(z,x),y]'\\
	&=& [[x,y],z] + [[y,z],x] + [[z,x],y] + \hbar[\phi(x,y),z] + \hbar[\phi(y,z),x]\\ 
	&& + \hbar[\phi(z,x),y]  + \hbar\phi([x,y],z) + \hbar\phi([y,z],x) + \hbar\phi([z,x],y)\\
	&& + \hbar^2\phi(\phi(x,y),z) + \hbar^2\phi(\phi(y,z),x) + \hbar^2\phi(\phi(z,x),y)\\
	&=& 0.
	\end{eqnarray*}
	The previous equation is usually interpreted as follows. We set $\delta \phi(x,y,z) = [\phi(x,y),z] + [\phi(y,z),x]+ [\phi(z,x),y]  + \phi([x,y],z) + \phi([y,z],x) + \phi([z,x],y)$ and also (assuming ${\rm char}\ \mathbb k\neq 2$) set $1/2[\phi,\phi](x,y,z) = \phi(\phi(x,y),z) + \phi(\phi(y,z),x) + \phi(\phi(z,x),y)$. Since the coefficient of the term $1/2[\phi,\phi]$ is quadratic in the ``infinitesimal'' deformation parameter $\hbar$, it follows that $[\bullet,\bullet]'$ is an infinitesimal deformation if and only if $\delta^2 \phi = 0$. This is the $2$-cocycle condition for Lie cochains $\phi: \mathfrak g\otimes \mathfrak g \longrightarrow \mathfrak g$ with coefficients in $\mathfrak g$.
	
	\subsection{$n$-Lie algebras} 
	
	The case of higher arity Lie algebras was first considered in \cite{Tak}. We use a slightly modified reasoning which leads to definitions that are more suitable for our purposes in the rest of the present article. Namely, we do not consider a homology theory and dualize it, but rather we direcly follow the deformation theory point of view of Gerstenhaber and Schack and adapt it to the higher arity case. We explicitly write the formulas for ternary brackets, since this generalizes directly to the $n$-ary case. The following treatment is closely related to that of Takhtajan in \cite{Tak}. 
	
	As before, we consider a Lie algebra with (ternary) bracket $[\bullet, \bullet, \bullet]: \mathfrak g\otimes \mathfrak g\otimes \mathfrak g\longrightarrow \mathfrak g$, see \cite{Fil}. We extend coefficients to the ring $\mathbb k[[\hbar]]$ and quotient out by the maximal ideal $(\hbar^2)$, with the intention of finding deformations that are infinitesimal of order $1$ in $\hbar$. We also set, for simplicity, $\phi(x,y,z) = [x,y,z]_1$, for the ternary bracket of order $1$. A deformed bracket is now required to satisfy the Filippov idenity (FI) on $x,y,z,w,u\in \mathfrak g$. That is, we have 
	\begin{eqnarray*}
	[[x,y,z]',w,u]' = [[x,w,u]',y,z]' + [x,[y,w,u]',z]' + [x,y, [z,w,u]']'.
	\end{eqnarray*}
	Using the definition of $[\bullet,\bullet,\bullet]'$ we obtain, as in the case of binary deformation, the following. For the left hand side of FI we have
	\begin{eqnarray*}
	[[x,y,z]',w,u]' &=& [[x,y,z],w,u] + \hbar[\phi(x,y,z),w,u] + \hbar\phi([x,y,z],w,u)\\ 
	&& + \hbar^2 \phi(\phi(x,y,z),w,u).
	\end{eqnarray*}
	For the right hand side of FI we have
	\begin{eqnarray*}
	\lefteqn{[[x,w,u]',y,z]' + [x,[y,w,u]',z]' + [x,y, [z,w,u]']'}\\
	 &=& [[x,w,u],y,z] + [x,[y,w,u],z] + [x,y, [z,w,u]] + \hbar[\phi(x,w,u),y,z]\\ && + \hbar[x,\phi(y,w,u),z] + \hbar[x,y,\phi(z,w,u)] + \hbar \phi(x,[y,w,u],z) + \hbar \phi(x,y,[z,w,u])\\
	&& +\hbar \phi([x,w,u],y,z) + \hbar^2\phi(\phi(x,w,u),y,z) + \hbar^2\phi(x,\phi(y,w,u),z) + \hbar^2\phi(x,y,\phi(z,w,u)).
	\end{eqnarray*}
\begin{sloppypar}
	Neglecting terms of order $2$, FI holds if and only if 
	\begin{eqnarray*}
		\delta^2 \phi(x,y,z,w,u) &:=& [\phi(x,y,z),w,u] + \phi([x,y,z],w,u) -[x,\phi(y,w,u),z]\\
		&&- [x,y,\phi(z,w,u)] -\phi([x,w,u],y,z) - \phi(x,[y,w,y],z) - \phi(x,y,[z,w,u]) \\
		&=& 0.
	\end{eqnarray*}
	 This is defined to be the $2$-cocycle condition for the $3$-Lie algebra $\mathfrak g$ with coefficients in $\mathfrak g$.
	\end{sloppypar}
	
	\subsection{$2$-cocycles and cohomology}
	
	Let $C^2_n(\mathfrak g; \mathfrak g)$, with $n = 2,3$ be the group of $2$-cochains with coefficients in $\mathfrak g$. In other words, $C^2_n(\mathfrak g; \mathfrak g)$ is the space of linear maps $\mathfrak g^{\otimes n} \longrightarrow \mathfrak g$. As usual, it is clear that the $2$-cocycle condition for $n$-Lie algebras, $n = 2,3$, is closed under addition and it determines an additive subgroup of cocycles $Z^2_n(\mathfrak g; \mathfrak g) \subset C^2_n(\mathfrak g; \mathfrak g)$ where $n=2,3$ depending on the arity of $\mathfrak g$.
	
	Let us define now the group of $1$-cochains $C^1_n(\mathfrak g; \mathfrak g) = {\rm Hom}_{\mathbb k}(\mathfrak g,\mathfrak g)$. We define the first differential $\delta^1 : C^1_n(\mathfrak g; \mathfrak g) \longrightarrow C^2_n(\mathfrak g; \mathfrak g) $ as follows. For $n=2$, we set $\delta^1f(x,y) = f([x,y]) - [f(x),y] - [x,f(y)]$, while for $n = 3$ we set $\delta^1f(x,y,z) =  f([x,y,z]) - [f(x), y, z] - [x,f(y), z] - [x,y, f(z)]$. 
	
	Clearly, while the $2$-cocycle condition is the obstruction for a bracket to be infinitesimally deformed, we have that the $1$-cocycle condition is the obstruction for a $1$-cochain to be a Lie algebra derivation of $\mathfrak g$, where the notion of derivation in the $n$-ary case is a straightforward generalization of the usual binary definition. We define $\delta^1 C^1_n(\mathfrak g; \mathfrak g) = B^2_n(\mathfrak g; \mathfrak g) $, the subgroup of coboundaries.
	
	Direct computations show that $\delta^2 \delta^1 = 0$, hence coboundaries 
	 are also $2$-cocycles. Setting $H^2_n(\mathfrak g; \mathfrak g) = Z^2_n (\mathfrak g; \mathfrak g)/ B^2_n(\mathfrak g; \mathfrak g)$ we have the second cohomology group of $\mathfrak g$ with coefficients in $\mathfrak g$. The binary case gives the usual cohomology of Lie algebras, i.e. the dual to the Eilenberg-Chevalley complex. The ternary case closely parallels the deformation theoretic interpretation of the binary Lie cohomology and gives a definition of ternary $2$-cohomology that suits well our purposes. 
	
	The definitions for the case $n=3$ extend easily to arbitrary $n$, therefore providing a notion of second cohomology group for $n$-Lie algebras. 
	
	Recall that two deformations of a $2$-Lie bracket are said to be equivalent if there exists an isomorphism between the two deformed Lie structures that fixes the zero degree bracket. For higher arity we define equivalent deformations mutatis mutandis. We have the following standard result. 
	
	\begin{theorem}\label{thm:Liedeform}
		Let $\mathfrak g$ be an $n$-Lie algebra. Then, infinitesimal deformations of $\mathfrak g$ are classified by the second cohomology group $H^2_n(\mathfrak g;\mathfrak g)$. 
	\end{theorem}
    \begin{proof}
    	The proof adapts the methods of the binary case to higher arities, and is therefore only sketched. From the discussion above, it is clear that infinitesimal deformations arise if and only if $[\bullet, \ldots, \bullet]' = [\bullet, \ldots, \bullet] + \hbar\phi(\bullet, \ldots, \bullet)$ for some $2$-cocycle $\phi$ of appropriate arity. This was seen explicitly for $n = 2,3$, while the general result for arbitrary $n$ is analogous. To complete the characterization one needs to prove that two extensions of $\mathfrak g$ are equivalent if and only if the corresponding $2$-cocycles differ by a term $\delta^1 f$, for some $f : \mathfrak g\longrightarrow \mathfrak g$. In fact, it is enough to show that $\phi$ defines the trivial deformation if and only if $\phi = \delta^1f$. This is seen directly, for arbitrary $n$. 
    \end{proof}
    
    \section{Preliminaries on self-distributive objects}\label{Prelim}
    
    In this section we recall some of the definitions that will be used throughout the article. This material can be found in \cite{AZ,EZ,CCES,ESZ}. 
    
    \begin{definition}
        {\rm 
        An {\it $n$-ary quandle}, or $n$-quandle for short, is a set $X$ with an $n$-ary map 
        $T: X^{\times n} \longrightarrow X$ satisfying the three axioms:
    \begin{itemize}
        \item
         It holds 
         $$T(T(x_1,\ldots x_n), x_{n+1}, \ldots, x_{2n-1}) = T(T(x_1, x_{n+1}, \ldots , x_{2n-1}), \ldots , T(x_n, x_{n+1}, \ldots x_{2n-1})),$$ for all $x_i\in X$. This property is called n-ary self-distributivity;
         \item 
         The map $T_{(x_2,\ldots , x_n)}: X\longrightarrow X$ obtained by multiplying on the right, i.e. $T_{(x_2,\ldots , x_n)}(x) := T(x,x_2,\ldots , x_n)$, is bijective for all choices of $(x_2,\ldots , x_n)\in X^{\times {(n-1)}}$;
         \item 
         For all $x\in X$, we have that $T(x,\ldots ,x) = x$.
    \end{itemize} 
        }
    \end{definition}
    
    \begin{remark}
        {\rm 
        The case $n=2$, which is known simply as ``quandle'' in the literature, has found profound applications in knot theory (see \cite{Joy,Mat,CJKLS, CEGS}). In fact, the axioms given above for $n=2$ are seen to be an algebraization of the Reidemeister moves (III, II and I, respectively). Here we chose to list the axioms in reversed order than usually found in the literature, because algebraically, a set with a map $T$ that satisfies only the first axiom is called a {\it shelf}, or {\it self-distributive} set, while if $T$ satisfies the first two axioms is said to be a {\it rack}. We will use $n$-shelf and $n$-rack to specify the arity of the operation $T$, if needed.
        }
    \end{remark}
    
    The definition of shelf, rack and quandle generalizes to symmetric monoidal categories as seen in \cite{ESZ,EZ}. The SD property, which is the one of main interests in the present article, takes the form of the following commutative diagram 
    \begin{center}
	\begin{tikzcd}
		&X^{\boxtimes n^2}\arrow{dl}[swap]{\shuffle_n} & &X^{\boxtimes (2n-1)}\arrow{ll}[swap]{\mathbb 1^{\boxtimes n}\boxtimes \Delta_n^{\boxtimes (n-1)}}\arrow {rd}{W\boxtimes \mathbb{1}^{\boxtimes (n-1)}} &  \\
		 X^{\boxtimes n^2}\arrow{dd}[swap]{W\boxtimes \cdots \boxtimes W} & & & & X^{\boxtimes n}\arrow{dd}{W}\\
		 & & & &\\
		 X^{\boxtimes n}\arrow{rrrr}[swap]{W}& & & &X 
		\end{tikzcd} 
	\end{center}
    where $W: X^{\boxtimes n} \longrightarrow X$ is a morphism in the symmetric monoidal category $(\mathcal C, \boxtimes)$, $\Delta$ is a comultiplication of $X$, and the morphism $\shuffle_n$ is obtained by a permutation of the factors (using the symmetry of the category). The specific form of $\shuffle_n$ is given in \cite{ESZ}. Observe that in this case $X$ is endowed both with an $n$-ary ``operation'', and with a (binary) comultiplication. The comultiplication is used to create $n$-ary comultiplications by compositions. In fact, the same theory where $\Delta$ is $n$-ary and not decomposable in co-operations of lower degree can be defined analogously, and it has not yet been explored, to the authors' knowledge. The need for a coalgebra structure arises from the fact that the set-theoretic case uses repetitions of the variables, which can be thought of as a comultiplication (diagonal) in the symmetric monoidal category of sets with cartesian product. 
    
    The main interest in $n$-ary SD structures (similarly to the binary set-theoretic case), is that they produce solutions to the Yang-Baxter equation (YBE). When the structure is a (categorical) rack, then these solutions are invertible, and they are therefore YB operators. See for instance \cite{EZ} for a general approach. In \cite{AZ}, which constitutes the starting point of the present article, the ternary approach was considered in the category of vector spaces and the YB operators studied were derived from Lie algebras, based on work found in \cite{CCES}. 
    
    We briefly recall the construction in \cite{AZ} that produces SD objects in the category of vector spaces from Lie algebras, and associated YB operators. Let $\mathfrak g$ be an $n$-Lie algebra over the ground field $\mathbb k$, and define $X := \mathbb k\oplus \mathfrak g$. 
	We define a comultiplication $\Delta$ (see also \cite{CCES}) as
	$$
	(a,x) \mapsto (a,x)\otimes (1,0) + (1,0) \otimes (0,x),
	$$
	and counit $\epsilon$ as 
	$$
	(a,x)\mapsto a.
	$$
	With these comultiplication and counit, $X$ is a coalgebra in the category of vector spaces. Iterating the comultiplication we can define a ternary co-operation $\Delta_3$. Observe that $\Delta$ (hence $\Delta_3$) is cocommutative. We will use this coalgebra structure on the space $X$ defined above throughout the article. However, we point out that the cohomology that we are going to define to study deformations of SD structures does not depend on $\Delta$, but it only uses the fact that we start from an SD object. The SD structure associated to an $n$-Lie algebra $\mathfrak g$ then is then $(X,\Delta, T)$, where $X$ and $\Delta$ are as above, and for all $a_i \in \mathbb k$ and $x_i \in \mathfrak g$, the map $T$ is defined on simple tensors as
	$$
	(a_1,x_1)\otimes \cdots \otimes (a_n,x_n) \mapsto (\prod_i a_i, \prod_{i\neq1} a_ix_1 + [x_1,\ldots ,x_n]),
	$$
	where $[\bullet, \ldots , \bullet]$ indicates the $n$-ary bracket of $\mathfrak g$. In fact, starting from an $n$-ary structure, one can define a $(2n-1)$-ary SD structure as it was shown in \cite{ESZ} for symmetric monoidal categories. The same procedure can be applied in this case. Therefore, we see that $n$-ary structures in the category of vector spaces are abundant.  In \cite{AZ} it is detailed how to derive a YB operator from a given TSD structure associated to a ternary Lie algebra. The generalization to the $n$-ary case is analogous. This will be reconsidered in Section~\ref{sec:YB_deform}, where we will show the main application of the cohomology theory introduced and studied in this article. 
    
	\section{SD cohomology with SD coefficients}\label{sec:TSD_coh}
	
	We define now a cohomology theory of $n$-SD structues with coefficients in $n$-SD objects that is suitable for the deformation theory of SD operations. The fundamental idea is the same as for $n$-Lie algebras above. We explicitly write the equations for the case $n=3$, since the whole theory generalizes directly to the case of arbitrary $n$ (including $n=2$). An $n$-ary SD structure will be denoted by $n$SD, but the binary and ternary cases, due to their importance in this article will also be called BSD and TSD, respectively.
	
	Let $(X,T,\Delta_3)$ be a TSD structure in the symmetric monoidal category of modules over a unital ring $\mathbb k$, where $T\in {\rm Hom}_{\mathbb k}(X^{\otimes 3},X)$ and $\Delta_3\in {\rm Hom}_{\mathbb k}(X,X^{\otimes 3})$.
	For $n=1,2,3$ we define the first and second cochain modules of $X$ with coefficients in $X$ as follows. Let ${\rm Hom}_{\mathbb k}(X^{\otimes (2n-1)}, X)$ denote the spaces of $\mathbb k$-linear maps between $X^{\otimes (2n-1)}$ and $X$. We define $C^1_{\rm TSD}(X;X)$, the first cochain module, to be the submodule of ${\rm Hom}_{\mathbb k}(X, X)$ consisting of ternary coderivations, where a ternary coderivation is a linear map $f: X\longrightarrow X$ satisfying $\Delta_3\circ f = (\mathbb 1^{\otimes 2}\otimes f + \mathbb 1\otimes f \otimes \mathbb 1  + f\otimes \mathbb 1^{\otimes 2})\circ \Delta_3$. This naturally generalizes the notion of binary coderivations between coalgebras to the ternary setting. Next, define $C^2_{\rm TSD}(X;X)$, the second cochain module, to be the submodule of ${\rm Hom}_{\mathbb k}(X^3, X)$ consisting of maps satisfying the condition $\Delta_3 \psi = (\psi\otimes T^{\otimes 2}+T\otimes \psi \otimes T + T^{\otimes 2}\otimes \psi)\shuffle \Delta_3^{\otimes 3}$, where $\shuffle = \bigl(\begin{smallmatrix}
	1 & 2& 3& 4 & 5& 6& 7& 8& 9\\
	1& 4& 7& 2& 5& 8& 3& 6& 9  
	\end{smallmatrix}\bigr)$, i.e. the permutation used to endow$X^{\otimes 3}$ with the natural ternary coalgebra structure induced by $\Delta_3$. Finally, we let $C^3_{\rm TSD}(X;X) = {\rm Hom}_{\mathbb k}(X^{\otimes 5}, X)$. 
	
	We define now differentials between cochain modules $\delta^1: C^1_{\rm TSD}(X;X) \longrightarrow C^2_{\rm TSD}(X;X)$ and $\delta^2: C^2_{\rm TSD}(X;X)\longrightarrow C^3_{\rm TSD}(X;X)$ as follows.
	
	For $f: X\longrightarrow X$, ternary coderivation of $X$, we set
	\begin{eqnarray*}
	\delta^1f(x\otimes y\otimes z) &=& f(T(x\otimes y\otimes z)) - T(f(x)\otimes y\otimes z) - T(x\otimes f(y)\otimes z) - T(x\otimes y\otimes f(z)).
	\end{eqnarray*}
    Notice that the condition $\delta^1 f = 0$ is the obstruction for $f$ to be a ternary derivation with respect to the ternary algebraic structure $T$. 
    
    For $\psi: X^{\otimes 3} \longrightarrow X$, coalgebra morphism between $X^{\otimes 3}$ and $X$, we set
    \begin{eqnarray*}
    \delta^2 \psi(x\otimes y\otimes z\otimes w\otimes u) &=& T(\psi(x\otimes y\otimes z)\otimes w\otimes u) + \psi(T(x\otimes y\otimes z)\otimes w\otimes u)\\
    && - \psi(T(x\otimes w^{(1)}\otimes u^{(1)})\otimes T(y\otimes w^{(2)}\otimes u^{(2)})\otimes T(z\otimes w^{(3)}\otimes u^{(3)}))\\
    &&- T(\psi(x\otimes w^{(1)}\otimes u^{(1)})\otimes T(y\otimes w^{(2)}\otimes u^{(2)})\otimes T(z\otimes w^{(3)}\otimes u^{(3)}))\\
    && - T(T(x\otimes w^{(1)}\otimes u^{(1)})\otimes \psi(y\otimes w^{(2)}\otimes u^{(2)})\otimes T(z\otimes w^{(3)}\otimes u^{(3)}))\\
    && -  T(T(x\otimes w^{(1)}\otimes u^{(1)})\otimes T(y\otimes w^{(2)}\otimes u^{(2)})\otimes \psi(z\otimes w^{(3)}\otimes u^{(3)})),
    \end{eqnarray*}
	where we have used Sweedler's notation for the comultiplication of $X$, in the form 
	$\Delta(x) = x^{(1)}\otimes x^{(2)} \otimes x^{(3)}$. The $2$-cocycle condition is shown diagrammatically in Figure~\ref{fig:2-cocy}. See also a similar condition (in multiplicative notation) in \cite{EZ}.
	
	\begin{figure}
	    \centering
	    $0 =
	    \begin{tikzpicture}[scale=0.50,baseline={([yshift=0cm]current bounding box.center)},vertex/.style={anchor=base,
    	circle,fill=black!25,minimum size=18pt,inner sep=2pt}]
	    \draw (-3,2) -- (0,0);
	    \draw (-1.5,2) -- (0,0);
	    \draw (0,2) -- (0,0);
	    \draw (3,2) -- (0,0);
	    \draw (1.5,2) -- (0,0);
	    \draw (0,0) -- (0,-2);
	    \draw[fill=black] (0,0) circle (5pt);
	    \node (a) at (-1.75,-.250) {$\delta^2(\psi)$}; 
	    \end{tikzpicture}
	    =
	    \begin{tikzpicture}[scale=0.50,baseline={([yshift=0cm]current bounding box.center)},vertex/.style={anchor=base,
    	circle,fill=black!25,minimum size=18pt,inner sep=2pt}] 
    	\draw (-3,3) -- (-2,2);
    	\draw (-2,3) -- (-2,2);
    	\draw (-1,3) -- (-2,2);
    	\draw (-2,2) -- (0,0);
    	\draw (0,3) -- (0,0);
    	\draw (2,3) -- (0,0);
    	\draw (0,0) -- (0,-2);
    	\draw[fill=black] (0,0) circle (5pt);
    	\draw[fill=black] (-2,2) circle (5pt);
    	\node (a) at (-.5,-.25) {$T$};
    	\node (a) at (-2.5,1.75) {$\psi$};
	    \end{tikzpicture}
	    +
	    \begin{tikzpicture}[scale=0.50,baseline={([yshift=0cm]current bounding box.center)},vertex/.style={anchor=base,
    	circle,fill=black!25,minimum size=18pt,inner sep=2pt}] 
    	\draw (-3,3) -- (-2,2);
    	\draw (-2,3) -- (-2,2);
    	\draw (-1,3) -- (-2,2);
    	\draw (-2,2) -- (0,0);
    	\draw (0,3) -- (0,0);
    	\draw (2,3) -- (0,0);
    	\draw (0,0) -- (0,-2);
    	\draw[fill=black] (0,0) circle (5pt);
    	\draw[fill=black] (-2,2) circle (5pt);
    	\node (a) at (-.5,-.25) {$\psi$};
    	\node (a) at (-2.5,1.75) {$T$};
	    \end{tikzpicture}
	    $\\
	    $
	    -
	    \begin{tikzpicture}[scale=0.50,baseline={([yshift=0cm]current bounding box.center)},vertex/.style={anchor=base,
    	circle,fill=black!25,minimum size=18pt,inner sep=2pt}] 
    	\draw (-4,3) -- (-3,2);
    	\draw (-3,3) -- (-3,2);
    	\draw (-2,3) -- (-3,2);
    	\draw (-3,2) -- (0,0);
        \draw (-1,3) -- (0,2);
        \draw (0,3) -- (0,2);
        \draw (1,3) -- (0,2);
        \draw (0,2) -- (0,0);
        \draw (4,3) -- (3,2);
    	\draw (3,3) -- (3,2);
    	\draw (2,3) -- (3,2);
    	\draw (3,2) -- (0,0);
    	\draw (0,0) -- (0,-2);
    	\draw[fill=black] (0,0) circle (5pt);
    	\draw[fill=black] (-3,2) circle (5pt);
    	\draw[fill=black] (3,2) circle (5pt);
    	\draw[fill=black] (0,2) circle (5pt);
    	\node (a) at (-.5,-.25) {$\psi$};
    	\node (a) at (-3.5,1.75) {$T$};
    	\node (a) at (-.5,1.75) {$T$};
    	\node (a) at (3.5,1.75) {$T$};
        \draw (0,6) -- (0,5);
        \draw[fill=black] (0,5) circle (5pt);
        \node (a) at (-1.75,5.25) {$\shuffle_3\circ \Delta_3$};
        \draw (0,5) -- (0,3);
        \draw (0,5) -- (-1,3);
        \draw (0,5) -- (-2,3);
        \draw (0,5) -- (-3,3);
        \draw (0,5) -- (-4,3);
        \draw (0,5) -- (1,3);
        \draw (0,5) -- (2,3);
        \draw (0,5) -- (3,3);
        \draw (0,5) -- (4,3);
	    \end{tikzpicture}
	    -
	    \begin{tikzpicture}[scale=0.50,baseline={([yshift=0cm]current bounding box.center)},vertex/.style={anchor=base,
    	circle,fill=black!25,minimum size=18pt,inner sep=2pt}] 
    	\draw (-4,3) -- (-3,2);
    	\draw (-3,3) -- (-3,2);
    	\draw (-2,3) -- (-3,2);
    	\draw (-3,2) -- (0,0);
        \draw (-1,3) -- (0,2);
        \draw (0,3) -- (0,2);
        \draw (1,3) -- (0,2);
        \draw (0,2) -- (0,0);
        \draw (4,3) -- (3,2);
    	\draw (3,3) -- (3,2);
    	\draw (2,3) -- (3,2);
    	\draw (3,2) -- (0,0);
    	\draw (0,0) -- (0,-2);
    	\draw[fill=black] (0,0) circle (5pt);
    	\draw[fill=black] (-3,2) circle (5pt);
    	\draw[fill=black] (3,2) circle (5pt);
    	\draw[fill=black] (0,2) circle (5pt);
    	\node (a) at (-.5,-.25) {$T$};
    	\node (a) at (-3.5,1.75) {$\psi$};
    	\node (a) at (-.5,1.75) {$T$};
    	\node (a) at (3.5,1.75) {$T$};
        \draw (0,6) -- (0,5);
        \draw[fill=black] (0,5) circle (5pt);
        \node (a) at (-1.75,5.25) {$\shuffle_3\circ \Delta_3$};
        \draw (0,5) -- (0,3);
        \draw (0,5) -- (-1,3);
        \draw (0,5) -- (-2,3);
        \draw (0,5) -- (-3,3);
        \draw (0,5) -- (-4,3);
        \draw (0,5) -- (1,3);
        \draw (0,5) -- (2,3);
        \draw (0,5) -- (3,3);
        \draw (0,5) -- (4,3);
	    \end{tikzpicture}
	    $\\
	    $
	    -
	    \begin{tikzpicture}[scale=0.50,baseline={([yshift=-.5cm]current bounding box.center)},vertex/.style={anchor=base,
    	circle,fill=black!25,minimum size=18pt,inner sep=2pt}] 
    	\draw (-4,3) -- (-3,2);
    	\draw (-3,3) -- (-3,2);
    	\draw (-2,3) -- (-3,2);
    	\draw (-3,2) -- (0,0);
        \draw (-1,3) -- (0,2);
        \draw (0,3) -- (0,2);
        \draw (1,3) -- (0,2);
        \draw (0,2) -- (0,0);
        \draw (4,3) -- (3,2);
    	\draw (3,3) -- (3,2);
    	\draw (2,3) -- (3,2);
    	\draw (3,2) -- (0,0);
    	\draw (0,0) -- (0,-2);
    	\draw[fill=black] (0,0) circle (5pt);
    	\draw[fill=black] (-3,2) circle (5pt);
    	\draw[fill=black] (3,2) circle (5pt);
    	\draw[fill=black] (0,2) circle (5pt);
    	\node (a) at (-.5,-.25) {$T$};
    	\node (a) at (-3.5,1.75) {$T$};
    	\node (a) at (-.5,1.75) {$\psi$};
    	\node (a) at (3.5,1.75) {$T$};
        \draw (0,6) -- (0,5);
        \draw[fill=black] (0,5) circle (5pt);
        \node (a) at (-1.75,5.25) {$\shuffle_3\circ \Delta_3$};
        \draw (0,5) -- (0,3);
        \draw (0,5) -- (-1,3);
        \draw (0,5) -- (-2,3);
        \draw (0,5) -- (-3,3);
        \draw (0,5) -- (-4,3);
        \draw (0,5) -- (1,3);
        \draw (0,5) -- (2,3);
        \draw (0,5) -- (3,3);
        \draw (0,5) -- (4,3);
	    \end{tikzpicture}
	    -
	    \begin{tikzpicture}[scale=0.50,baseline={([yshift=-.5cm]current bounding box.center)},vertex/.style={anchor=base,
    	circle,fill=black!25,minimum size=18pt,inner sep=2pt}] 
    	\draw (-4,3) -- (-3,2);
    	\draw (-3,3) -- (-3,2);
    	\draw (-2,3) -- (-3,2);
    	\draw (-3,2) -- (0,0);
        \draw (-1,3) -- (0,2);
        \draw (0,3) -- (0,2);
        \draw (1,3) -- (0,2);
        \draw (0,2) -- (0,0);
        \draw (4,3) -- (3,2);
    	\draw (3,3) -- (3,2);
    	\draw (2,3) -- (3,2);
    	\draw (3,2) -- (0,0);
    	\draw (0,0) -- (0,-2);
    	\draw[fill=black] (0,0) circle (5pt);
    	\draw[fill=black] (-3,2) circle (5pt);
    	\draw[fill=black] (3,2) circle (5pt);
    	\draw[fill=black] (0,2) circle (5pt);
    	\node (a) at (-.5,-.25) {$T$};
    	\node (a) at (-3.5,1.75) {$T$};
    	\node (a) at (-.5,1.75) {$T$};
    	\node (a) at (3.5,1.75) {$\psi$};
        \draw (0,6) -- (0,5);
        \draw[fill=black] (0,5) circle (5pt);
        \node (a) at (-1.75,5.25) {$\shuffle_3\circ\Delta_3$};
        \draw (0,5) -- (0,3);
        \draw (0,5) -- (-1,3);
        \draw (0,5) -- (-2,3);
        \draw (0,5) -- (-3,3);
        \draw (0,5) -- (-4,3);
        \draw (0,5) -- (1,3);
        \draw (0,5) -- (2,3);
        \draw (0,5) -- (3,3);
        \draw (0,5) -- (4,3);
	    \end{tikzpicture}
	    $
	    \caption{String diagram for the $2$-cocycle condition of a $2$-cochain $\psi$. The strings are permuted using the natural shuffling morphism for the category of modules.}
	    \label{fig:2-cocy}
	\end{figure}
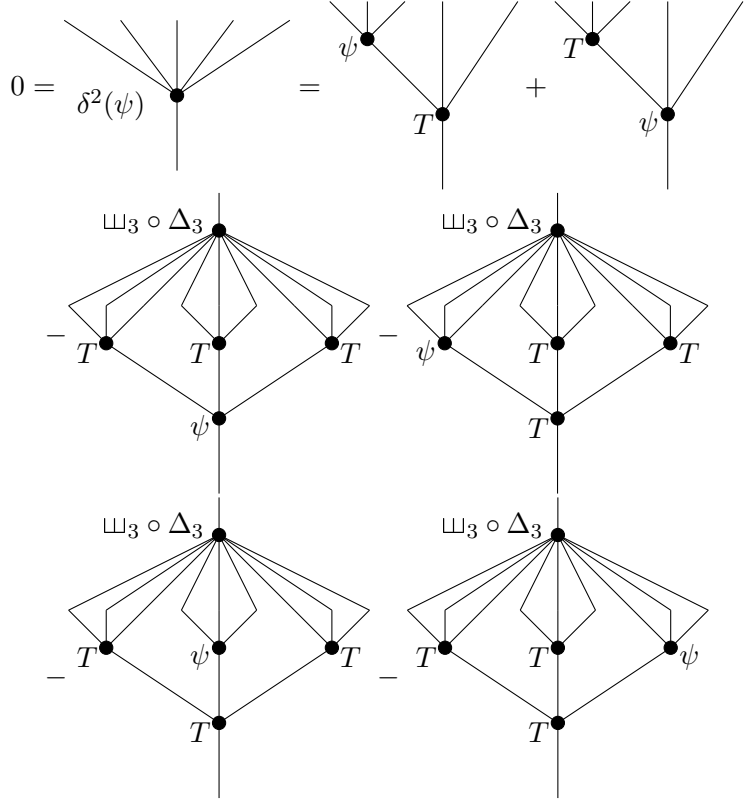
	
	\begin{lemma}\label{lem:differentials}
		The differential $\delta^1$ maps $1$-cochains into $C^2_{\rm TSD}(X;X)$. Moreover,  the composition of differentials $\delta^2\circ \delta^1$ is trivial.
	\end{lemma}
\begin{proof}
	This is a long direct computation which we postpone to the Appendix.
\end{proof}

As a consequence of Lemma~\ref{lem:differentials}, the following is well posed.
\begin{definition}\label{def:SDcohomology}
	{\rm 
	Let $X$ be an $n$-SD structure over the unital ring $\mathbb k$. Following traditional notation, we set $Z^2_{n\rm SD}(X;X) = {\rm ker}\ \delta^2$, and $B^2_{n\rm SD}(X;X) = {\rm im}\ \delta^1$. The second cohomology group of $X$ with coefficients in $X$ is then the quotient $H^2_{n\rm SD}(X;X) = Z^2_{n\rm SD}(X;X)/B^2_{n\rm SD}(X;X)$. For the ternary case we will use the notation $H^2_{\rm TSD}(X;X)$, and similarly for cocycles and coboundaries. 
}
\end{definition}

Straightforward modifications generalize the previous definitions to the case of cohomology of $X$ with coefficients in a TSD object $A$ different from $X$ itself. The proof of Lemma~\ref{lem:differentials} is easily adapted to this case. 

\begin{remark}
    {\rm 
    We observe that Definition~\ref{def:SDcohomology} does not explicitly require $\Delta$ to be the comultiplication given above, as long as $(T,\Delta,\epsilon)$ is an SD object. In fact, Lemma~\ref{lem:differentials} does not need any special consideration regarding the actual form of $\Delta$ to hold, but only uses the SD property (see Appendix). This fact is going to play a role in the simultaneous deformations considered in Section~\ref{sec:sim_deform}. Until then, we will consider the special form for $\Delta$ described above.
    }
\end{remark}	

\begin{lemma}\label{lem:coderivation}
	Let $V$ be a vector space and let $X = \mathbb k \oplus V$ be  the coalgebra with comultiplication $\Delta(a,x) = (a,x)\otimes (1,0) + (1,0)\otimes (0,x)$. Then, a linear map $f: X\longrightarrow X$ is an $n$-ary coderivation if and only if it is of type $f(a,x) = (0,g(x))$ for some linear map $g: V \longrightarrow V$. 
\end{lemma}
\begin{proof}
	We consider the binary case, since the $n$-ary one is substantially similar.
	This is a direct inspection of the coderivation condition $\Delta f(a,x) = (f\otimes \mathbb 1 + \mathbb 1\otimes f) \Delta(a,x)$, where $f(a,x) = f_0(a,x) + f_1(a,x)$ for maps $f_0 : X \longrightarrow \mathbb k$ and $f_1: X\longrightarrow V$. From the definition of $\Delta$, we have that the LHS of the coderivation condition is 
	$$\Delta f(a,x) = f(a,x)\otimes (1,0) + (1,0)\otimes (0,f_1(a,x)),$$
	 while for the RHS we have 
	 \begin{eqnarray*}
	 (f\otimes \mathbb 1 + \mathbb 1\otimes f)\Delta(a,x) &=&   (f\otimes \mathbb 1 + \mathbb 1\otimes f)[(a,x)\otimes (1,0) + (1,0)\otimes (0,x)]\\
	 &=& f(a,x)\otimes (1,0) + f(1,0)\otimes (0,x) + (a,x)\otimes f(1,0) + (1,0)\otimes f(0,x).
	 \end{eqnarray*}
 Therefore, the coderivation condition becomes
 $$
 (1,0)\otimes (0,f_1(a,x)) = f(1,0)\otimes (0,x) + (a,x)\otimes f(1,0) + (1,0)\otimes f(0,x).
 $$
 Since this equation should hold for all $x\in V$, it follows that $f(1,0) = 0$, since otherwise the RHS woud have a component in $V\otimes X$ which could not vanish. It follows that the coderivation condition gives 
 $$
 (1,0)\otimes (0,f_1(a,x)) = (1,0)\otimes f(0,x)
 $$
 which, writing $(a,x)$ as a sum in $X$, holds if and only if $af_1(1,0) + f_1(0,x) = f_0(0,x) + f_1(0,x)$. Therefore, we have found that it must hold $af(1,0) = f_0(0,x)$ which forces $f_0(0,x) = 0$, completing the proof. 
\end{proof}

Applying Lemma~\ref{lem:coderivation}, it follows that the space of $1$-cochains for the SD cohomology consists of linear maps $f: X\longrightarrow X$ that factors through the Lie algebra $\mathfrak g$ as 
\begin{center}
	\begin{tikzcd}
		X\arrow[rr,"f"]\arrow[d,"\pi_2"] & & X\\
		\mathfrak g\arrow[rr,"\tilde f"] & & \mathfrak g\arrow[u,"\iota_2"]
	\end{tikzcd}
\end{center}
for some $\tilde f$, where $\pi_2$ is the projection on the second summand of $X = \mathbb k \oplus \mathfrak g$, and $\iota_2$ is the obvious injection $\mathfrak g \hookrightarrow X = \mathbb k \oplus \mathfrak g$. In particular, this provides a direct criterion for an SD $2$-cocycle to be nontrivial as follows. Lemma~\ref{lem:coderivation} applied to a $1$-cochain $f$, along with the definition of differential $\delta_{SD}^1f$, shows that $\delta_{SD}^1f$ is trivial on $\mathbb k\otimes X$, and morover it has image on $\mathfrak g\otimes \mathfrak g$. Therefore, any $2$-cocycle $\phi$ that is nontrivial on a simple tensor of type $(1,0)\otimes (a,x)$, or such that $\phi$ has nontrivial image in $\mathbb k\otimes X + X\otimes \mathbb k$, is not cobounded. It therefore represents a nontrivial cohomological class in $H^2_{\rm SD}(X;X)$. In fact, it turns out that there are examples of $2$-cocycles of this sort. For instance, the Lie algebra Bianchi IV over the complex field has two nontrivial cohomology classes of this type, cf. Section~\ref{sec:examples}. 
\begin{remark}
    {\rm 
    We point out that a formulation of higher order cochains (along with their corresponding cohomology groups) might be attempted along the same lines, although we have not explicitly tried so. However, we are not aware of current potential applications of cohomology groups of degree larger than $3$. For degree $3$, J. Stasheff has pointed to us the possible interesting application of Gerstenhaber bracket type obstructions (elements of the third cohomology groups) to extending the deformations to higher orders. This is a particularly interesting perspective when considering the Yang-Baxter operator deformations discussed in Section~\ref{sec:YB_deform} below, especially with the target of producing link invariants, from arbitrarily deformed YB operators. We have not yet explored this promising line of research. 
    }
\end{remark}

	\section{Infinitesimal deformations of SD structures}\label{sec:TSDdeform}
	
	Once again, let us consider the ternary case for simplicity, even though the results of this section generalize to the $n$-ary case following an analogous reasoning. We show that the second cohomology group $H^2_{\rm TSD}(X;X)$ of a TSD object with coefficients in itself classifies the infinitesimal deformations of $X$, as for associative and Lie algebras. 
	
	The notion of infinitesimal deformation for a TSD object follows the same lines as for $n$-Lie algebras above. Namely, given a TSD object $X$ over a field $\mathbb k$, we extend coefficients of $X$ to $\mathbb k[[\hbar]]/(\hbar^2)$, and let us denote it as $X_\hbar$. We define a new ternary operation on $X_\hbar$ as a linear map $T' = T + \hbar \psi_1$, where $\psi_1: X^{\otimes 3} \longrightarrow X$. We say that $T'$ is an infinitesimal deformation of $T$ if $T'$ turns $X_\hbar$ into a TSD object, where the comultiplication of $X_\hbar$ is that of $X$, $\Delta_3$, extended to $X_\hbar$ by linearity. Two infinitesimal deformations $T'$ and $T''$ of $T$ are said to be equivalent if there exists an ismorphism between $(X_\hbar,T')$ and $(X_\hbar, T'')$ fixing the underlying zero degree structure.
	
	\begin{theorem}\label{thm:TSDdeform}
		Let $X$ be a TSD object over $\mathbb k$. Then, the cohomology group $H^2_{\rm TSD}(X;X)$ classifies the infinitesimal deformations of $X$.
	\end{theorem}
\begin{proof}
	First, we show that $T + \hbar\psi$ defines an infinitesimal deformation of $T$ if and only if $\psi\in Z^2_{\rm TSD}(X;X)$. Let $x\otimes y\otimes z\otimes w\otimes u\in X^{\otimes 5}$ be a simple tensor. The left hand side of TSD condition on  $T' = T + \hbar\psi$ gives
	\begin{eqnarray*}
	T'(T'(x\otimes y\otimes z)\otimes w\otimes u) &=& T(T(x\otimes y\otimes z)\otimes w\otimes u) + \hbar T(\psi(x\otimes y\otimes z)\otimes w\otimes u)\\
	&& + \hbar \psi(T(x\otimes y\otimes z)\otimes w\otimes u) + \hbar^2 \psi(\psi(x\otimes y\otimes z)\otimes w\otimes u).
	\end{eqnarray*}
    The right hand side gives 
    \begin{eqnarray*}
    \lefteqn{T'(T'(x\otimes w^{(1)}\otimes u^{(1)})\otimes T'(y\otimes w^{(2)}\otimes u^{(2)})\otimes T'(z\otimes w^{(3)}\otimes u^{(3)}))}\\
     &=& T(T(x\otimes w^{(1)}\otimes u^{(1)})\otimes T(y\otimes w^{(2)}\otimes u^{(2)})\otimes T(z\otimes w^{(3)}\otimes u^{(3)}))\\
     &&+ \hbar \psi(T(x\otimes w^{(1)}\otimes u^{(1)})\otimes T(y\otimes w^{(2)}\otimes u^{(2)})\otimes T(z\otimes w^{(3)}\otimes u^{(3)}))\\
     && + \hbar T(\psi(x\otimes w^{(1)}\otimes u^{(1)})\otimes T(y\otimes w^{(2)}\otimes u^{(2)})\otimes T(z\otimes w^{(3)}\otimes u^{(3)}))\\
     && + \hbar T(T(x\otimes w^{(1)}\otimes u^{(1)})\otimes \psi(y\otimes w^{(2)}\otimes u^{(2)})\otimes T(z\otimes w^{(3)}\otimes u^{(3)}))\\
     && + T(T(x\otimes w^{(1)}\otimes u^{(1)})\otimes T(y\otimes w^{(2)}\otimes u^{(2)})\otimes \psi(z\otimes w^{(3)}\otimes u^{(3)}))\\
     && + \hbar^2(\cdots) + \hbar^3(\cdots), 
    \end{eqnarray*}
    where the terms in brackets correspond to terms where $\psi$ appears twice and thrice. Since higher powers of $\hbar$ are zero, equating terms of degree one we find that $T'$ satisfies the TSD condition if and only if $\psi$ satisfies the $2$-cocycle condition. A similar direct computation also shows that $T'$ is a coalgebra morphism if and only if $\psi$ satisfies the condition $\Delta_1 \psi = (\psi\otimes T^{\otimes 2}+T\otimes \psi \otimes T + T^{\otimes 2}\otimes \psi)\shuffle \Delta_3^{\otimes 3}$.
    
      To complete the proof, we need to show that a deformation is equivalent to the trivial deformation if and only if the corresponding $2$-cocycle $\psi$ is cobounded by some coderivation $f$. In fact, suppose that $T' = T + \hbar \psi$ is equivalent to the trivial deformation. Then, by definition, there exists a $\mathbb k[[\hbar]]/(\hbar^2)$-linear map $g$ of the form $g = \mathbb 1 + \hbar f$ such that $g$ is a coalgebra morphism and a TSD morphism. The latter condition enforces that $g(T+\hbar \psi)(g^{-1}(x)\otimes g^{-1}(y)\otimes g^{-1}(z)) = T(x\otimes y\otimes z)$. Expanding $g^{-1}$ as a geometric series in $\hbar$ and grouping terms of order one, we see that $\psi = \delta^1 f$. Imposing that $g$ is a coalgebra morphism and, again, taking terms of degree $1$ shows that $f$ is a coderivation. It follows that $[\psi] = 0$ in $H^2_{\rm TSD}(X;X)$. A similar reasoning also shows that coboundaries give rise to trivial deformations, completing the proof.
\end{proof}

	\section{From Lie algebra cohomology to SD cohomology}\label{sec:Lie_SD}
	
	In this section we construct homomorphisms between the second cohomology groups of a Lie algebra and the corresponding SD structures. Observe here that we consider maps from binary and ternary Lie algebras to both binary and ternary SD structures, along with maps from ternary Lie algebras to ternary SD structures. In fact, as seen in \cite{ESZ}, composing $n$-SD structures produces $(2n-1)$-SD structures, therefore from binary Lie algebras one gets $3$-SD structures as well. The results of this section can be generalized to maps (in cohomology) between $n$-Lie algebras and $n$-SD as well as $(2n-1)$-SD structures. In fact, it is reasonable to assume that the same procedure can be iterated deriving maps to higher order SD structures, even though we have not explicitly verified that this is the case. This is a relevant question which we leave to future work.
	
	\subsection{Binary Lie algebras} 
	
	Let $\frak g$ be a Lie algebra, and let $X = \mathbb C\oplus \frak g$ denote the associated BSD structure. Then, we define a map $\Psi^2 : C^2_{\rm Lie}(\frak g; \frak g) \longrightarrow C^2_{\rm BSD}(X;X)$ as follows. On $2$-chains $(a,x)\otimes (b,y)$ we set $\Psi^2(\phi)(a,x)\otimes (b,y) = (0,\phi(x,y))$, where $\phi$ denotes a Lie algebra $2$-cocycle with coefficients in $\frak g$. 
	
	\begin{proposition}\label{pro:LietoBSD}
		Let $\frak g$, $X$ and $\Psi^2$ be as above. Then $\Psi^2$ induces a well defined map between cohomology groups $H^2_{\rm Lie}(\frak g; \frak g)$ and $H^2_{\rm TSD}(X;X)$. 
	\end{proposition}

	In what follows we use the notation $\phi((a,x)\otimes (b,y)) = (\phi_0((a,x)\otimes (b,y)),  \phi_1((a,x)\otimes (b,y)))$, for $2$-cochains, and when no confusion arises, we simply write $a + x$ instead of $(a,x)$ to indicate the direct sum $\mathbb k\oplus \mathfrak g$. So, $\phi_1((a,x)\otimes (b,y))$ will denote $(0,\phi_1((a,x)\otimes (b,y)))$, for instance. 

	\begin{definition}
		{\rm 
		Let $X$ denote a BSD structure corresponding to a Lie algebra $\frak g$. Let $\phi\in C^2_{\rm BSD} (X;X)$.
		 Then, if $\hat\psi(a,x)\otimes (b,y) = (0,\psi(x,y))$ for all $(a,x)$ and $(b,y)$, we say that the cohomology class $[\hat\psi]$ is {\it special}, and we simply denote it by the map $\psi$.
		}
	\end{definition}
		
	The following lemma will be key in proving the main result of the section. 
	\begin{lemma}\label{lem:special}
		Let $\mathfrak g$ be a Lie algebra and $X$ the associated BSD object. Then, any $2$-cochain $\phi \in C^2(X;X)$ is such that $\phi((a,x)\otimes (b,y)) = \phi_1((a,x)\otimes (b,y))$,  and moreover $\phi_1(\mathbb k\otimes X) = 0$. If in addition $\mathfrak g$ has trivial center, any $2$-cocycle $\phi \in Z^2_{\rm BSD}(X;X)$ is special. 
	\end{lemma}
	\begin{proof}
	Let $\phi$ denote a $2$-cochain representing a cohomology class $[\phi]\in H^2_{\rm BSD}(X;X)$. Then, by definition of $C^2_{\rm BSD}(X;X)$ we have that $\phi$ satisfies the equation
	$$
	\Delta \phi ((a,x)\otimes (b,y)) = [\phi\otimes q + q\otimes \phi](\mathbb 1\otimes \tau \otimes \mathbb 1)(\Delta\otimes \Delta)((a,x)\otimes (b,y)),
	$$ where $q:X \times X \rightarrow X$ is given by $q((a,x)\otimes (b,y))=(ab,bx+[x,y])$.
	Using the definitions, the $2$-cochain condition for $\phi$ becomes
	\begin{eqnarray*}
	\lefteqn{\phi((a,x)\otimes (b,y))\otimes (1,0) + (1,0)\otimes \phi_1((a,x)\otimes (b,y))}\\
	 &=& \phi((a,x)\otimes (b,y))\otimes (1,0) + \phi((a,x)\otimes (1,0))\otimes q((1,0)\otimes (0,y))\\
	 && + \phi((1,0)\otimes (b,y))\otimes (0,x) + \phi((1,0)\otimes (1,0))\otimes (0,[x,y])\\
	 && + (ab,bx+[x,y])\otimes \phi((1,0)\otimes (1,0)) + (a,x)\otimes \phi((1,0)\otimes (0,y))\\
	 && + b(1,0)\otimes \phi((0,x)\otimes (1,0)) + (1,0)\otimes \phi((0,x)\otimes (0,y)),
	\end{eqnarray*}
	and this is readily rewritten as 
	\begin{eqnarray*}
		\lefteqn{(1,0)\otimes \phi_1((a,x)\otimes (b,y))}\\
		&=& + \phi((1,0)\otimes (b,y))\otimes (0,x) + \phi((1,0)\otimes (1,0))\otimes (0,[x,y])\\
		&& + (ab,bx+[x,y])\otimes \phi((1,0)\otimes (1,0)) + (a,x)\otimes \phi((1,0)\otimes (0,y))\\
		&& + b(1,0)\otimes \phi((0,x)\otimes (1,0)) + (1,0)\otimes \phi((0,x)\otimes (0,y)).
	\end{eqnarray*}
	Let us take $y = 0$ in the previous equation. Then, the RHS has a component of type $(0,x)\otimes (c,z)$, coming from $(ab,bx+[x,y])\otimes \phi((1,0)\otimes (1,0))$, which is not present on the LHS. This forces $\phi((1,0)\otimes (1,0)) = 0$. The $2$-cochain condition, therefore, becomes
	\begin{eqnarray*}
		(1,0)\otimes \phi_1((a,x)\otimes (b,y)) &=& \phi((1,0)\otimes (0,y))\otimes (0,x) + (a,x)\otimes \phi((1,0)\otimes (0,y))\\
		&& + b(1,0)\otimes \phi((0,x)\otimes (1,0)) + (1,0)\otimes \phi((0,x)\otimes (0,y)),
	\end{eqnarray*}
	where we have also simplified the term $ \phi((1,0)\otimes (b,y))\otimes (0,x)$ to $ \phi((1,0)\otimes (0,y))\otimes (0,x) $ by virtue of the vanishing of $\phi((1,0)\otimes (1,0))$. But then this forces $ \phi((1,0)\otimes (0,y))$ to be trivial, since otherwise $(a,x)\otimes \phi((1,0)\otimes (0,y))$ would have a component of type $(a,x)\otimes (c,z)$ which has no counterpart in the LHS. We have obtained
		\begin{eqnarray*}
		(1,0)\otimes \phi_1((a,x)\otimes (b,y)) &=& b(1,0)\otimes \phi((0,x)\otimes (1,0)) + (1,0)\otimes \phi((0,x)\otimes (0,y)).
	\end{eqnarray*}
	Now, let us consider the case $y = 0$. From this we obtain that 
	$$
	b(1,0)\otimes (0,\phi_1((0,x)\otimes(1,0))) = b(1,0)\otimes \phi((0,x)\otimes(1,0)),
	$$
	which shows that $\phi_0(X\otimes \mathbb k) = 0$. Similarly, setting $b=0$ gives that $\phi((0,x)\otimes (0,y)) = \phi_1((0,x)\otimes (0,y))$, showing that $\phi_0 = 0$ on $\mathfrak g\otimes \mathfrak g$ as well. This completes the first part of the statement. 
	
	Let us now suppose that $\mathfrak g$ has trivial center, and that $\phi$ is a $2$-cocycle. Applying the $2$-cocycle condition on simple tensors of type $(0,x)\otimes (1,0)\otimes (0,z)$ we find
	\begin{eqnarray*}
	\lefteqn{\phi(q((0,x)\otimes (1,0))\otimes (0,z))  + q(\phi((0,x)\otimes (1,0))\otimes (0,z))}\\
	&=& q(\phi((0,x)\otimes (0,z))\otimes(1,0)) + q( \phi((0,x)\otimes (1,0))\otimes q((1,0)\otimes (0,z)))\\
	&&+ q(q((0,x)\otimes (0,z))\otimes \phi((1,0)\otimes (1,0))) + q((0,x)\otimes \phi((1,0)\otimes (1,0))) \\
	&& + \phi(q((0,x)\otimes (0,z))\otimes (1,0)) +  \phi((0,x)\otimes q((1,0)\otimes (0,z)))\\
	&=& \phi((0,x)\otimes (0,z)) + \phi((0,[x,z])\otimes (1,0)),
	\end{eqnarray*}
	where we have used the conditions on $\phi$ proved in the first part of the proof, and the definition of $q$, in particular the fact that $q((1,0)\otimes (0,x)) = 0 $ for all $x\in \mathfrak g$. The previous equality then is seen to reduce to
	$$
	[\phi_1((0,x)\otimes (1,0)),z] = \phi_1((0,[x,z])\otimes (1,0)),
	$$
	which holds for arbitrary $x, z \in \mathfrak g$. We proceed similarly now on simple tensors of type $(0,x)\otimes (0,y)\otimes (1,0)$, for which the $2$-cocycle condition takes the form
	\begin{eqnarray*}
	\phi(q((0,x)\otimes (0,x))\otimes (1,0)) + q(\phi((0,x)\otimes (0,y))\otimes (1,0)) &=& 
	q(\phi((0,x)\otimes (1,0))\otimes (0,y)) \\
	&&+ q((0,x)\otimes \phi((0,y)\otimes (1,0))\\
	&& + \phi((0,x)\otimes (0,y)).
	\end{eqnarray*}
   Applying the condition obtained from the previous case, this equation is seen to reduce to $0 = (0, [x,\phi_1((0,y)\otimes (1,0))])$, which holds for all $x,y\in \mathfrak g$. Therefore, $\phi_1$ maps $\mathfrak g\otimes \mathbb k$ in the center of $\mathfrak g$, which is trivial by hypothesis. Consequently, $\phi_1(\mathfrak g\otimes \mathbb k) = 0$, and $\phi$ is special, completing the proof.  
	\end{proof}

	\begin{theorem}\label{thm:LieBSD}	
		Let $\frak g$ be a Lie algebra and let $X$ be its associated BSD object. Then $\Psi^2$ induces a monomorphism $H^2_{\rm Lie}(\frak g;\frak g) \longrightarrow H^2_{\rm BSD}(X;X)$ in cohomology.
		Suppose, in addition, that $\mathfrak g$ has trivial center (e.g. it is semisimple), then $H^2_{\rm Lie}(\frak g;\frak g) \cong H^2_{\rm BSD}(X;X)$. 
	\end{theorem}
	\begin{proof}

	Let us indicate the map $H^2_{\rm Lie}(\frak g;\frak g) \longrightarrow H^2_{\rm BSD}(X;X)$ induced by $\Psi^2$ in cohomology (from Proposition~\ref{pro:LietoBSD}) by the same symbol $\Psi^2$. We start by showing that $\Psi^2$ is injective. Let $\phi$ be a representative of a cohomology class in $H^2_{\rm Lie}(\frak g;\frak g)$ and suppose that $[\Psi^2(\phi)] = [0]$.  Then, there exists a $1$-cochain $f: X\longrightarrow X$ such that $\Psi^2(\phi)= \delta^1f$. Observe that, by definition of $\Psi^2$,  $\Psi^2(\phi)$ is special. Therefore, it follows that $\delta_{SD}^1$ is special as well, since it would not otherwise cobound $\Psi^2(\phi)$. Consequently, we can write $\delta^1 f(a,x)\otimes (b,y) = (0, g(a,x)\otimes (b,y))$. A direct computation shows that $g(a,x)\otimes (b,y) = \delta^1_{\rm Lie} f_1x\otimes y$. Then, this means that $\phi = \delta^1_{\rm Lie} f_1$, which in turn implies that $[\phi] = [0]$ in the Lie cohomology group. This shows that $\psi^2$ is injective.
	
	 Let us now consider the second part of the statement. Surjectivity is derived from the fact that cohomology classes of $X$ are special, i.e. Lemma~\ref{lem:special} where $\mathfrak g$ has trivial center. Let $[\psi]$ be a cohomology class in $H^2_{\rm BSD}(X;X)$. Then, the fact that $\psi$ is special implies that $\psi((a,x)\otimes (b,y)) = \psi_1((0,x)\otimes (0,y))$. Then, we can define $\sigma : \mathfrak g^{\otimes 2} \longrightarrow \mathfrak g$ by the assignment $g(x\otimes y) : = \psi_1((0,x)\otimes (0,y))$. It follows that $\Psi^2(\sigma) = \psi$, showing that $\Psi^2$ is surjective.  
	\end{proof}

\begin{remark}
	{\rm 
We observe that Theorem~\ref{thm:LieBSD} has a deep meaning in terms of deformation theory. It substantially states, in fact, that the deformation theory of SD structures coincides with the deformation theory of the Lie algebra that generates the SD object. 
}
\end{remark}
	
	We consider now the case of binary Lie algebras and TSD structure associated to it via composition of binary self-disrtibutive operations. Let $\mathfrak g$ denote a (binary) Lie algebra, and let $X$ denote the corresponding TSD structure arising from  $\mathfrak g$. Then, for a $2$-cocycle $\phi \in Z^2_{\rm Lie}(\mathfrak g, \mathfrak g)$, we define a map $\Theta(\phi) : X^{\otimes 3} \longrightarrow X$ according to the assignment
	$$
	\Theta^2(\phi)((a,x)\otimes (b,y)\otimes (c,z)) = (0, b\phi(x\otimes z) +c\phi(x\otimes y) + [\phi(x,y),z] + \phi([x,y],z)). 
    $$
	
	\begin{proposition}\label{prop:2LietoTSD}
		Let $\mathfrak g$, $X$ and $\Theta^2$ be as above. Then $\Theta^2$ induces a map between cohomology groups $H^2_{\rm Lie}(\mathfrak g; \mathfrak g)$ and $H^2_{\rm TSD}(X;X)$. 
	\end{proposition}
\begin{proof}
 We need to show that $\Theta^2$ maps a $2$-cocycle $\phi$ to a $2$-cochain of $X$, and that it satisfies the $2$-cocycle condition. This means that $\Theta^2(\phi)$ must satisfy  the property $\Delta_3 \Theta^2(\phi) = [\Theta^2(\phi)\otimes T^{\otimes 2}+T\otimes \Theta^2(\phi) \otimes T + T^{\otimes 2}\otimes \Theta^2(\phi)]\shuffle \Delta_3^{\otimes 3}$, and moreover $\delta^2 \Theta^2(\phi)  = 0$, where $\delta^2$ is the differential of Section~\ref{sec:TSDdeform}. Both properties are easily verified as follows. Let $[\bullet,\bullet]_\psi$ denote the Lie alegbra structure obtained as a deformation of the bracket $[\bullet,\bullet]$ of $\mathfrak g$, as an application of Theorem~\ref{thm:Liedeform}. From Theorem~8.1 in \cite{ESZ}, the map $T_\phi: X_\hbar^{\otimes 3} \longrightarrow X_\hbar$ associated to $[\bullet,\bullet]_\phi$ turns $X_\hbar$ into a TSD object. Direct inspection shows that $T_\phi = T + \hbar \Theta^2 (\phi)$. Then, applying Theorem~\ref{thm:TSDdeform} we conclude that $\Theta^2(\phi)$ satisfies the required properties. To complete the proof, we need to show that $\Theta^2(\delta^1 f)$ is a coboundary for any $f: \mathfrak g \longrightarrow \mathfrak g$. In fact, given such a $1$-cochain $f$, we define $g : X\longrightarrow X$ as $g(a,x) = (0,f(x))$. Then, we see that $\delta^1_{\rm TSD} g ((a,x)\otimes (b,y)\otimes (c,z) = (0,bf([x,y]) + cf([x,y]) + f([[x,y],z]) - b[f(x),z] - c[f(x),y] - [[f(x),y],z] - c[x,f(y)] - [[x,f(y)],z] - b[x,f(z)] - [[x,y],f(z)])$. It is a direct verification to see that the previous term coincides with $\Theta^2(\delta^1_{\rm Lie}f)$ evaluated on $(a,x)\otimes (b,y)\otimes (c,z)$, up to an overall negative sign. It follows that setting $\Theta^1(f)(a,x) = (0,-f(x))$ gives a cochain map at level $1$ and $2$ between Lie cohomology and TSD cohomology, and $\Theta^2$ induces a well defined map in cohomology. 
\end{proof}
	
We now show that the morphism of Proposition~\ref{prop:2LietoTSD} is injective.

\begin{proposition}
	The map $\Theta^2$ is an injection.
\end{proposition}
\begin{proof}
	Let $\phi$ be a Lie $2$-cocycle and suppose that $[\Theta^2(\phi)] = 0$ in $H^2_{\rm TSD}(X;X)$. Let $f$ be a $1$-cochain that cobounds $\Theta^2(\phi)$. As in the proof of Theorem~\ref{thm:LieBSD}, since $\Theta^2$ maps $2$-cochains to special ones, $\delta^1 f$ needs to be special as well, which in turn implies that $f(a,x) = (0,g(x))$ for some $g$. Then, the equation $\delta^1_{\rm TSD} f((a,x)\otimes (b,y)\otimes (c,z)) = \Theta^2(\phi)((a,x)\otimes (b,y)\otimes (c,z)) $ gives
	\begin{eqnarray*}
	&&bg([x,z]) + cg([x,y]) + g([[x,y],z]) - b[g(x),z] - c[g(x),y] - [[g(x),y],z]\\
	&&- b[x,z] - c[x,g(y)] - [[x,g(y)],z] -[x,g(z)] - c[x,y] - [[x,y],g(z)]\\
	&=& b\phi(x\otimes z) + c\phi(x\otimes y) + [\phi(x\otimes y),z] + \phi([x,y]\otimes z).
	\end{eqnarray*}
Considering the case $(1,x)\otimes (1,y)\otimes (1,0)$ reduces to the equation 
$$
g([x,y]) - [g(x),y]  - [x,g(y)] = \phi(x\otimes y),
$$
which shows that $\phi = \delta^1_{\rm Lie} g$, therefore completing the proof. 
\end{proof}
	
	\subsection{$3$-Lie algebras}
	
	A construction similar to that of $2$-Lie algebras is applicable in the case of $3$-Lie algebras and their corresponding TSD objects. The main difference is that due to the fact that the $3$-bracket is not obtained as the composition of lower degree ones, the chain map involved is simpler than that of Proposition~\ref{prop:2LietoTSD}. 
	
	Let $\mathfrak g$ be a $3$-Lie algebra and let $X$ denote the corresponding TSD object. Given a $2$-cocycle $\phi : \mathfrak g^{\otimes 3} \longrightarrow \mathfrak g$ of $\mathfrak g$, we define the map $\Lambda^2(\phi): X^{\otimes 3} \longrightarrow X$ according to 
	$$
	\Lambda^2(\phi)((a,x)\otimes (b,y)\otimes (c,z)) = (0,\phi(x\otimes y\otimes z)).
	$$
	
	\begin{proposition}\label{prop:3LietoTSD}
		Let $\mathfrak g$, $X$ and $\Lambda^2$ be as above. Then $\Lambda^2$ induces a map between cohomology groups $H^2_{3{\rm Lie}} (\mathfrak g; \mathfrak g)$ and $H^2_{\rm TSD}(X;X)$. 
	\end{proposition}
\begin{proof}
	The proof is analogous to that of Proposition~\ref{prop:2LietoTSD}. First one shows that the TSD structure induced by a deformed $3$-Lie bracket can be written as $T + \hbar \Lambda^2(\phi)$, and then one sees that $\Lambda^2$ can be extended to a cochain map at $n = 1, 2$, for some $\Lambda^1$. The corresponding computations are easier than those in Proposition~\ref{prop:2LietoTSD}. 
\end{proof}

Recall (\cite{Fil}) that for a ternary Lie algebra $\frak g$, we say that the $3$-bracket $[\bullet, \bullet, \bullet]$ has trivial center if $[x,y,z] = 0$ for all $x$ implies either $y = 0$, or $z = 0$. The following result is completely analogous to Theorem~\ref{thm:LieBSD} and it is proved following similar techniques adapted to ternary operations. 

\begin{theorem}
	Let $\frak g$ be a ternary Lie algebra and let $X$ denote the associated TSD structure. Then, there is a monomorphism 
	$\Lambda^2: H^2_{3 \rm Lie}(\frak g ;\frak g) \longrightarrow H^2_{\rm TSD}(X;X)$.
	Moreover, if $\frak g$ has trivial center, $\Lambda^2$ gives an isomorphism $H^2_{3 \rm Lie}(\frak g ;\frak g) \cong H^2_{\rm TSD}(X;X)$. 
\end{theorem}

\section{Coalgebra deformations}\label{sec:Coalgebradeformation}

We consider now the problem of deforming the coalgebra structure of a coalgebra $(X,\Delta)$. In the present section, the map $T$ of a TSD object does not appear, and the deformations only refer to the comultiplication and counit structures. We will consider the simultaneous deformation of coalgebra and SD structures afterwards. 

Let $\Delta' = \Delta + \hbar \psi$ denote a deformed comultiplication, and let $\epsilon' = \epsilon + \hbar \iota$ be a deformed counit. Then, we have the following. 

\begin{lemma}\label{lem:deformedcoass}
	Let $(X,\Delta, \epsilon)$ be a coalgebra over $\mathbb k$, and let $\Delta'$, $\epsilon'$ be deformations over $\mathbb k'$ as above. Then, $(X',\Delta',\epsilon')$ is a coalgebra if and only if 
	\begin{eqnarray}
		(\mathbb 1\otimes \Delta)\psi + (\mathbb 1\otimes \psi)\Delta = (\Delta\otimes \mathbb 1)\psi + (\psi\otimes \mathbb 1)\Delta,\label{eqn:coass}
	\end{eqnarray}
	and 
	\begin{eqnarray}
	(\epsilon\otimes \mathbb 1)\psi + (\iota \otimes \mathbb 1)\Delta = 0\label{eqn:leftcounit} \\
	(\mathbb 1\otimes \epsilon)\psi + (\mathbb 1\otimes \iota)\Delta = 0. \label{eqn:rightcounit}
	\end{eqnarray}
\end{lemma}
\begin{proof}
		We need to verify the coassociativity property for $\Delta'$. We have
	\begin{eqnarray*}
		(\mathbb 1\otimes \Delta') \Delta' = (\mathbb 1\otimes \Delta) \Delta + \hbar[(\mathbb 1\otimes \Delta)\psi + (\mathbb 1\otimes \psi)\Delta] + \hbar^2 (\mathbb 1\otimes \psi)\psi  \\
		(\Delta'\otimes \mathbb 1) \Delta' = 	(\Delta\otimes \mathbb 1)  \Delta + \hbar [(\Delta \otimes \mathbb 1)\psi + (\psi \otimes \mathbb 1)\Delta] + \hbar^2 (\psi \otimes \mathbb 1)\psi.
	\end{eqnarray*}
	Equating the two results, using the fact that $\Delta$ is coassociative and considering that $\hbar^2 = 0$, we obtain that $\Delta'$ is coassociative if and only if equation~(\ref{eqn:coass}) holds. 
	
	Let us now consider the counit axiom, i.e. $(\epsilon'\otimes \mathbb 1) \Delta' =(\mathbb 1\otimes \epsilon')\Delta' = \mathbb 1$. We have 
	\begin{eqnarray*}
	(\epsilon'\otimes \mathbb 1) \Delta'  = (\epsilon\otimes \mathbb 1) \Delta + \hbar [(\epsilon\otimes \mathbb 1)\psi + (\iota \otimes \mathbb 1)\Delta] + \hbar^2 (\iota \otimes\mathbb 1)\psi 
	\end{eqnarray*}
which, as before, is equal to $\mathbb 1$ if and only if the terms of degree $1$ in $\hbar$ vanish, i.e. if and only if Equation~(\ref{eqn:leftcounit}) holds. Similar considerations apply for Equation~(\ref{eqn:rightcounit}).
\end{proof}

\begin{example}\label{ex:comdeform}
	{\rm 
	Let $X = \mathbb k \oplus V$, where $V$ is a vector space. Let $\Delta(1) = 1\otimes 1$ and $\Delta (x) = x\otimes 1 + 1\otimes x$,  extended by linearity, and let $\epsilon (1 ) = 1$, $\epsilon(x) = 0$ for all $x\in V$. Then $(X,\Delta,\epsilon)$ is a coalgebra. Let $\psi(x) := x\otimes x$ and let $\iota = 0$. One verifies by direct computation that $\Delta' = \Delta + \hbar \psi$ and $\epsilon' = \epsilon$ satisfy Equations~(\ref{eqn:coass}--\ref{eqn:rightcounit}) and therefore provide a coalgebra deformation of $(X,\Delta,\epsilon)$.
}
\end{example} 
 
With the result of Lemma~\ref{lem:deformedcoass} in mind, we now want to construct a cohomology group for coalgebras that classifies infinitesimal deformations of the coalgebra structure. We suppose that the counit $\epsilon$ is not deformed here. Moreover, we assume that $\Delta$ is cocommutative. 

Let us define a cochain complex in low dimensions as follows. For $n = 1,2,3$ we set $C^n_{\rm CA}(X;X) = {\rm Hom}_{\mathbb k} (X,X^{\otimes n})$. The, we define $\delta^1: C^1_{\rm CA}(X;X) \longrightarrow C^2_{\rm CA}(X;X)$ according to $f\mapsto \Delta f - [f\otimes \mathbb 1 + \mathbb 1\otimes f]\Delta$. Observe that $\delta^1f$ entails how far the homomorphism $f$ is from being a coderivation with respect to $\Delta$. In other words, $\delta^1 f = 0 $ if and only if $f$ is a coderivation. Next, we define $\delta^2 : C^2_{\rm CA}(X;X) \longrightarrow C^3_{\rm CA}(X;X)$ as $\psi \mapsto (\mathbb 1\otimes \Delta)\psi + (\mathbb 1\otimes \psi)\Delta - (\Delta\otimes \mathbb 1)\psi - (\psi\otimes \mathbb 1)\Delta$. As expected, the composition $\delta^2\delta^1$ vanishes. 

\begin{lemma}\label{lem:vanishcoass}
	With the notation of the previous paragraph, we have $\delta^2\delta^1 = 0$. 
\end{lemma}
\begin{proof}
	By direct computation, for $f\in Hom_{\mathbb k}(X,X)$, one obtains
	\begin{eqnarray*}
	\delta^2\delta^1(f) &=& (\mathbb 1\otimes \Delta)\Delta f - (\mathbb 1\otimes \Delta)(f\otimes \mathbb 1)\Delta - (\mathbb 1\otimes f)\Delta + [\mathbb 1\otimes (\Delta f)]\Delta \\
	&& - [\mathbb 1\otimes ([f\otimes 1]\Delta)]\Delta - [\mathbb 1\otimes ([\mathbb 1\otimes f])\Delta]\Delta - (\Delta\otimes \mathbb 1)\Delta f + (\Delta\otimes \mathbb 1)[(f\otimes \mathbb 1)\Delta]\\
	&& + (\Delta\otimes \mathbb 1)[(\mathbb 1\otimes f)\Delta] - [(\Delta f)\otimes \mathbb 1]\Delta
	 + [([f\otimes \mathbb 1]\Delta)\otimes \mathbb 1]\Delta + [([\mathbb 1\otimes f]\Delta)\otimes \mathbb 1]\Delta. 
	\end{eqnarray*}
    Using cocommutativity of $\Delta$, the previous quantity vanishes, and it therefore follows that $\delta^2\delta^1 = 0$.
\end{proof}

Applyig Lemma~\ref{lem:vanishcoass} it follows that we can define $H^2_{\rm CA}(X;X) = {\rm ker}(\delta^2)/{\rm im}(\delta^1)$ as usual. This is the second cohomology group of the coalgebra $(X,\Delta)$. 

\begin{theorem}
	Let $(X,\Delta,\epsilon)$ be a coalgebra. Then, the infinitesimal deformations of $\Delta$ ($\epsilon$ is not deformed here)  are classified by the coalgebra cohomology group $H^2_{\rm CA}(X;X)$. 
\end{theorem}
\begin{proof}
	The proof is similar to the TSD case considered before. Applying Lemma~\ref{lem:deformedcoass} with $\iota = 0$ one sees that $2$-cocycles are equivalent to deformations, while to show that trivial deformations correspond to $2$-cocycles that cobound, one expands the isomorphisms between the trivial deformation and the given deformation in a geometric series with respect to $\hbar$. 
\end{proof}

\section{Simultaneous deformations} \label{sec:sim_deform} 

As it was shown in \cite{AZ}, it is a problem of relevance to produce deformations of the TSD structure $(X,T)$ by simultaneously deforming both $T$ and $\Delta$. This requires that the new $T'$ satisfies the TSD property with respect to the deformed $\Delta'$, rather than $\Delta$, as considered above. Moreover, $T'$ needs to be a morphism of coalgebras with respect to $(X',\Delta')$. These considerations require a modification of the equations used above, although the base philosophy is identical. Once again, the passage to $n$-ary SD structures is deduced without excessive issues from the ternary one shown here, although the equations do become convoluted. 

We proceed as follows. First, we consider deformations of the coalgebra structure of $(X,\Delta,\epsilon)$ as in Section~\ref{sec:Coalgebradeformation}, and then we determine the conditions that a deformation of $T$ needs to satisfy in order to produce a TSD object. This requires substantially two conditions. First, $T'$ needs to satisfy the SD condition, then $T'$ needs to be a coalgebra morphism with respect to $\Delta'$. Observe that in both conditions, one needs to use $\Delta'$ and therefore the considerations derived in Section~\ref{sec:TSDdeform} need to be modified. We start with binary deformations, as these involve much simpler conditions, and composing deformed binary operations we can obtain deformed ternary operations as well. 

\begin{lemma}\label{lem:deformedSD}
	Let $(X,\Delta,q)$ be a binary SD object, and let $\Delta'$ be such that $(X',\Delta')$ is a coalgebra. Let $\shuffle$ indicate the shuffle map corresponding to binary self-distributivity.  Then $q' = q + \hbar \phi$ satisfies the SD condition with respect to $(X',\Delta')$ if and only if 
	\begin{eqnarray*}
		\lefteqn{q(\phi\otimes \mathbb 1) + \phi(q\otimes \mathbb 1)}\\
		&=& q[\phi\otimes q+q\otimes \phi]\shuffle (\mathbb 1^{\otimes 2}\otimes \Delta) + \phi(q\otimes q)\shuffle (\mathbb 1^{\otimes 2}\otimes \Delta) + q(q\otimes q)\shuffle (\mathbb 1^{\otimes 2}\otimes \psi).
	\end{eqnarray*}
	Moreover, $q'$ is a morphism of coaglebras if and only if 
	\begin{eqnarray*}
		\Delta \phi + \psi q
		&=& [\phi\otimes q + q\otimes \phi](\mathbb 1\otimes \tau \otimes \mathbb 1)(\Delta\otimes \Delta)\\
		&& + (q\otimes q)(\mathbb 1\otimes \tau \otimes \mathbb 1)(\psi\otimes \Delta + \Delta\otimes \psi),
	\end{eqnarray*}
	where $\tau$ indicates the switching map.
\end{lemma}

We provide, next, the conditions for deformations of ternary structures. A similar statement for $n$-ary SD structures is obtained by appropriately changing the number of tensorands and the arities of $T$ and $\Delta$. We just provide the ternary case for simplicity of notation. 

\begin{lemma}\label{lem:deformedTSD}
	Let $(X,\Delta,T)$ be a TSD object and let $\Delta'$ denote a deformation of the coalgebra structure. Then, $T' = T + \hbar \phi$ is a deformatioin of $T$ if and only if the following two conditions hold 
	\begin{eqnarray*}
	T(\phi\otimes \mathbb 1^{\otimes 2}) + \phi (T\otimes \mathbb 1^{\otimes 2}) &=& \phi[T^{\otimes 3}\shuffle (\mathbb 1^{\otimes 3}\otimes \Delta_3^{\otimes 2})]\\
	&& + T[(\phi\otimes T^{\otimes 2} + T\otimes \phi\otimes T + T^{\otimes 2}\otimes \phi)\shuffle(\mathbb 1^{\otimes 3}\otimes \Delta_3^{\otimes 2})]\\
	&& T[T^{\otimes 3}\shuffle (\mathbb 1^{\otimes 3}\otimes [(\mathbb 1\otimes \Delta)\psi + (\mathbb 1\otimes \psi)\Delta]^{\otimes 2})],
	\end{eqnarray*}

\begin{eqnarray*}
	(\mathbb 1\otimes \psi)\Delta T + (\mathbb 1\otimes \Delta)\psi T + \Delta_3\phi &=& 
	[\phi\otimes T^{\otimes 2} + T\otimes \phi\otimes T + T^{\otimes 2}\otimes \phi] \tilde \shuffle \Delta_3^{\otimes 3} \\
	&& + T^{\otimes 3}\tilde \shuffle [(\mathbb 1\otimes \Delta)\psi + (\mathbb 1\otimes \psi)\Delta]^{\otimes 3},
	\end{eqnarray*}
where $\shuffle$ is the permutation associated to TSD property, and $\tilde \shuffle$ is the permutation that gives the natural coalgebra structure on $X^{\otimes 3}$. 
\end{lemma}

\begin{remark}\label{rmk:smiultaneous}
	{\rm 
We observe that simultaneous deformations of $q$ or $T$, i.e. a binary or a ternary SD operation, can be interpreted as cohomology classes of $H^2_{\rm SD}(X_\psi;X_\psi)$, where $X_\psi$ is considered with the comultiplication on $X$ obtained by deforming $\Delta$ by $\psi$. So, the coalgebra cohomology parametrizes the SD second cohomology group of $X$. The structure so obtained is an SD object in the category of vector spaces in the sense of \cite{ESZ,AZ}.
}
\end{remark}

We point out that the simultaneous deformations of this section naturally generalize to higher arity SD operations the cohomology introduced in \cite{CCES}, based on the deformation theory approach of \cite{MS}. In fact, in \cite{CCES}, the cochain complexes contain direct sums corresponding to the coalgebra deformation and the SD deformation. In the present work, such an approach to put together the two deformations in a single complex would work as well, and the proofs of the well definedness of the second cohomology group substantially reduce the Lemma~\ref{lem:deformedTSD}. We therefore pose the following.

\begin{definition}\label{def:Sim_Co}
    {\rm 
    Let $(X,\Delta,q)$ denote a SD object. Then the second simultaneous cohomology group of $X$, denoted by the symbol $H^2_{\rm \Sigma SD}(X;X)$ is defined to be the group
    $$
    H^2_{\rm \Sigma SD}(X;X) := \bigoplus_{[\psi]\in H^2_{\rm CA}(X;X)}H^2_{SD}(X_\psi,X_\psi),
    $$
    where $\psi$ varies over the $2$-cocycles representatives of the coalgebra cohomology of $(X,\Delta)$. 
    }
\end{definition}

\begin{remark}
    {\rm 
    We note that when $\psi = 0$ is the trivial deformation, one recovers the deformation of SD structures of Section~\ref{sec:TSD_coh}, where the coalgebra structure is not deformed.  
    }
\end{remark}

\section{Examples and Computations}\label{sec:examples}

In this section we consider examples of nontrivial SD cohomology classes, and provide some computations for certain low-dimensional Lie algebras. 

\subsection{SD deformations}

First, we consider deformations of SD objects where the comultiplication is fixed. For simplicity of computation we focus on the binary case and we consider the SD structures arising from all the $3$-dimensional Lie algebras over the real and complex fields. 

We now provide the computations for the SD second cohomology groups of the SD objects associated to the binary Lie algebras in the Bianchi classification for the real and complex ground fields, where the comultiplication is not deformed. This accounts to having $\psi = \iota = 0$ in the previous computations. The results are summarised in Table~\ref{tab:coh_Bianchi} and Table~\ref{tab:coh_Bianchicomplex}, for the real and complex cases, respectively.

\begin{center}
	\begin{table}[h!]
		\begin{tabular}{ ||c | c  || }
			\hline
			Real Lie algebra &  Second Cohomology Group \\ 
			\hline 
			Bianchi II (Heisenberg) &  $\mathbb R^{\oplus 7}$  \\ 
			\hline 
			Bianchi III&  $\mathbb R$  \\
			\hline
			Bianchi IV &  $\mathbb R^{\oplus 6}$  \\
			\hline 
			Bianchi V& $\mathbb R$  \\
			\hline 
			Bianchi VI& $\mathbb R$ \\
			\hline 
			Bianchi VII & $\mathbb R^{\oplus 4}$ \\
			\hline 
			Bianchi VIII ($\mathfrak{sl}_2(\mathbb R)$) &  0 \\
			\hline 
			Bianchi IX ($\mathfrak{so}(3)$) & 0 \\
			\hline 
		\end{tabular}
		\caption{Cohomology of SD obsjects associated to real $3$D Lie algebras}   
		\label{tab:coh_Bianchi}
	\end{table}
\end{center}

\begin{center}
	\begin{table}[h!]
		\begin{tabular}{ ||c | c  || }
			\hline
			Complex Lie algebra &  Second Cohomology Group \\ 
			\hline 
			Bianchi II (Heisenberg) &  $\mathbb C^{\oplus 6}$ \\  
			\hline 
			Bianchi III&  $\mathbb C$ \\
			\hline
			Bianchi IV & $\mathbb C^{\oplus 6}$   \\
			\hline 
			Bianchi V& $\mathbb C$ \\
			\hline 
		Bianchi VI  & $\mathbb C^4$ \\
			\hline 
		Bianchi VIII ($\mathfrak{sl}_2(\mathbb C)$)&  0\\
			\hline 
		\end{tabular}
		\caption{Cohomology of SD obsjects associated to complex $3$D Lie algebras}   
		\label{tab:coh_Bianchicomplex}
	\end{table}
\end{center}

 We observe that, as expected, the computation of the cohomology groups for the simple algebras (i.e. $\mathfrak {sl_2}$ and $\mathfrak {so}(3)$) are trivial, since they are simple, and consequently have trivial center. In fact, known results (see for instance \cite{Jacob}) in Lie theory show that semisimple Lie algebras have trivial cohomology and, in addition, since they have trivial center it follows that the the corresponding SD cohomology is trivial as well from Theorem~\ref{thm:LieBSD}.
 However, niplotent and solvable algebras present nontrivial cohomology, which means that they have nontrivial deformations of the associated SD objects. 

\subsection{Simultaneous deformations}

We now consider the case of simultaneous deformations, where again we focus on the binary case for computational simplicity. We first give examples of simultaneously deformed structures for $\frak{sl}_2(\mathbb C)$ and the Heisenberg Lie algebra. We recall that the main objective of introducing simultaneous deformations is, in practice, to produce nontrivial deformations of the SD objects associated to Lie algebras with trivial center, so that our main interest here lies in the deformations of $\frak{sl}_2(\mathbb C)$, for which we explicitly compute the simultaneous cohomology $H^2_{\rm \Sigma SD}$. 

\begin{example}\label{ex:sl2}
	{\rm 
Let $\frak{sl}_2(\mathbb C)$ denote the special linear algebra of order $2$ over the complex numbers. 	Recall that this is a $3$-dimensional Lie algebra whose generators we indicate as $x, y, z$ and Lie bracket specified by $[x,y] = z$, $[z,y] = -2y$ and $[z,x] = 2x$. Let $q$ denote the corresponding BSD structure on $X = \mathbb C\oplus \frak{sl}_2(\mathbb C)$. We consider now simultaneous deformations of $\Delta$ and $q$. Following Section~\ref{sec:sim_deform} this accounts to finding $\phi$ and $\psi$ satisfying the conditions such that $\Delta' = \Delta + \hbar \psi$ and $q' = q + \hbar \phi$ define an SD object structure on $X'$, which is $X$ with coefficients in $\mathbb C[[\hbar ]]/ (\hbar^2)$. Let us set, for clarity, $e^1 = (1,0), e^2 = (0,x), e^3 = (0,y)$ and $e^4= (0,z)$, so that $2$-cochains $\phi : X\otimes X\longrightarrow X$ are given by the expression $\phi(e^i\otimes e^j) = \alpha^{ij}_ke^k$, where Einstein summation convention is understood. Similarly, for $\psi : X\longrightarrow X\otimes X$ we set $\psi(e^j) = \beta^{j}_{kl} e^k\otimes e^l$. 
Then, a direct computer aided calculation shows that the set of maps $\phi$ correspond to matrices
%
%
%
$$
A^T_\phi = \left(
\begin{matrix}
2\beta^3_{34} + \alpha_3^{31} & & &\\
\alpha_4^{24} + \alpha_4^{42} &-2\beta^3_{34}   & & \frac{\alpha_3^{41}}{2}\\
\alpha_4^{34} + \alpha_4^{43}& & 2\beta^3_{34}  & \frac{\alpha_2^{41}}{2}\\
 \alpha_4^{44} &-\alpha_2^{41} & -\alpha_3^{41}  &\\
&  4\beta^3_{34}  + \alpha_3^{31}& & -\frac{\alpha_3^{41}}{2}\\
&\alpha_4^{24} + \alpha_4^{42} & &\\
-\alpha_4^{44}& 2\alpha_4^{34} + \alpha_4^{34} &  -\alpha_4^{24} & -\alpha_4^{32} \\
2\alpha_4^{24}+ 2\alpha_4^{42}& -2\alpha_2^{42} + \alpha_4^{44} & -\alpha_3^{42} & \alpha_4^{24}\\
& & \alpha_3^{31} &-\frac{\alpha_2^{41}}{2} \\
\alpha_4^{44}& -\alpha_4^{34} & 2\alpha_4^{24} + \alpha_4^{42} &  \alpha_4^{32}\\
& & \alpha_4^{34} + \alpha_4^{34} &\\
-2\alpha_4^{34} - 2\alpha_4^{34}& -\alpha_2^{43} & \alpha_2^{42} + \alpha_4^{44} &  \alpha_4^{34}\\
& \alpha_2^{41} &  \alpha_3^{41}&2\beta^3_{34}  + \alpha_3^{31}\\
-2\alpha_4^{24} - 2\alpha_4^{42}&  \alpha_2^{42} &  \alpha_3^{42}& \alpha_4^{42}\\
2\alpha_4^{34}+ 2\alpha_4^{34}& \alpha_2^{43} & -\alpha_2^{42}  & \alpha_4^{34} \\
& & & \alpha_4^{44}
\end{matrix}
\right)
$$
while the maps $\psi$ are characterized as matrices
$$
B_\psi = \left( 
\begin{matrix}
-2\beta^3_{34}  - \alpha_3^{31}& -\alpha_4^{24} - \alpha_4^{42} & -\alpha_4^{34} - \alpha_4^{34} & -\alpha_4^{44} \\
 & -2\beta^3_{34}  - \alpha_3^{31} &  &  \\
&  & -2\beta^3_{34}  - \alpha_3^{31} &  \\
& & & -2\beta^3_{34}  - \alpha_3^{31} \\
& -2\beta^3_{34}  - \alpha_3^{31}& & \\
& -\alpha_2^{41}& & \\
& \frac{\alpha_3^{41}}{2} & -\frac{\alpha_2^{41}}{2} & \\
& \beta^3_{34} & &-\frac{\alpha_2^{41}}{2}  \\
& & -2\beta^3_{34}  - \alpha_3^{31} & \\
& \frac{\alpha_3^{41}}{2}  & -\frac{\alpha_2^{41}}{2}  & \\
& & \alpha_3^{41}& \\
& & \beta^3_{34} & \frac{\alpha_3^{41}}{2}  \\
&  & &-2\beta^3_{34}  - \alpha_3^{31} \\
& \beta^3_{34} & &-\frac{\alpha_2^{41}}{2}  \\
& & \beta^3_{34}  &  \frac{\alpha_3^{41}}{2} \\
& & & 2\beta^3_{34} 
\end{matrix}
\right),
$$
where we have omitted the zeros in both matrices, and $A_\phi$ has been transposed to better fit the page. 
The total number of free complex parameters is $14$, as it can be seen directly by the matrices of $\phi$ and $\psi$. In this example, the counit deformation, i.e. the map $\iota$ appearing in Section~\ref{sec:sim_deform}, is non-trivial as well. 
}
\end{example}

\begin{example}
	{\rm 
Let us now consider the $3$-dimensional Heisenberg algebra over the complex numbers, which we here denote as $\mathcal H_3(\mathbb C)$. Recall that this is generated by three elements $x, y, z$ and its Lie bracket is defined by  $[x,y] = z$, $[x,z] = [y,z] = 0$. Basis vectors for $X = \mathbb C\oplus \mathcal H_3$ are again $e^1 = (1,0), e^2 = (0,x), e^3 = (0,y)$ and $e^4 = (0,z)$ as before, and we will indicate $\phi$ and $\psi$ as in Example~\ref{ex:sl2}, while $\iota(e^k) = \iota_k\in \mathbb R$ for all $k=1,2,3,4$. Once again omitting the zeros, for simplicity, we have 
$$
A^T_\phi = \left(
\begin{matrix}
 \alpha_4^{31}& & &\\
 \alpha_1^{12}&   & &\beta^4_{43} \\
\alpha_1^{13}&  &   & -\beta^4_{24} \\
 \alpha_4^{44}&  &  & \\
& \alpha_4^{31} &  & \alpha_4^{21} \\
& \alpha_1^{12}& & \alpha_4^{22}\\
-\alpha_4^{44}&\alpha_1^{13} - \alpha_2^{32} & \alpha_1^{12} - \alpha_3^{32}& \alpha_4^{23} \\
& \alpha_3^{43} + \alpha_4^{44} & -\alpha_3^{42} & \alpha_1^{12} - \alpha_4^{42}\\
 &  & \alpha_4^{41} &\alpha_4^{31} \\
\alpha_4^{44}& -\alpha_2^{43}& \alpha_2^{33}  & \alpha_4^{32} \\
&  & \alpha_1^{13}  & \alpha_4^{33}\\
& -\alpha_3^{43}  & -\alpha_3^{43} + \alpha_4^{44}  & \alpha_1^{13} - \alpha_4^{43}  \\
& \alpha_2^{43} &  & \alpha_4^{41}\\
&   & \alpha_3^{42} & \alpha_4^{42} \\
& &  \alpha_3^{43} & \alpha_4^{43} \\
& & & \alpha_4^{44} 
\end{matrix}
\right)
$$
as well as 
$$
B_\psi = \left( 
\begin{matrix}
-\alpha_4^{41}& -\alpha_1^{12}  & -\alpha_1^{13} & -\alpha_4^{44}  \\
 & -\alpha_4^{41} &  &  \\
&  & -\alpha_4^{41}&  \\
& & & -\alpha_4^{41} \\
& -\alpha_4^{41}& & \\
&2\beta^4_{24} & & \\
&\beta^4_{43} & \beta^4_{24} & \\
&\beta^4_{44} - \beta^2_{24} & -\beta^3_{24} & \beta^4_{24}  \\
& & -\alpha_4^{41}& \\
& \beta^4_{43} & \beta^4_{24} & \\
& & 2\beta^4_{43} & \\
&-\beta^2_{34} & \beta^4_{44} - \beta^4_{34}  & \beta^4_{43}  \\
&  & &  -\alpha_4^{41}\\
& \beta^2_{24}  & \beta^3_{24} & \beta^4_{24}  \\
& \beta^2_{34} & \beta^3_{34} & \beta^4_{43} \\
& \beta^2_{44}& \beta^3_{44} & \beta^4_{44} 
\end{matrix}
\right).
$$
The counit deformation is nontrivial too, and it takes the form 
$$\iota = (\alpha_4^{41}, \alpha_1^{12}, \alpha_1^{13}, \alpha_4^{44}).$$
The previous matrices completely characterize the deformations of the SD struture obtained from the Heisenberg algebra $\mathcal H_3$, as well as the  coalgebra structure $(X,\Delta,\epsilon)$. We have a total of $26$ free parameters. 
}
\end{example}

\begin{remark}
	{\rm 
It is well known that semisimple Lie algebras have trivial cohomology group $H^2_{\rm Lie}(\mathfrak g;\mathfrak g)$, see \cite{Jacob} for example. Since they also have trivial center, from Theorem~\ref{thm:LieBSD} it follows that the BSD cohomology group of semisimple Lie algebras, e.g. $\mathfrak{sl}_2$, is trivial. Then, the $2$-cocycles computed above are all cobounded when $\psi = 0$ (i.e. the coalgebra structure is not deformed). However, in the case of simultaneous deformations where $\psi \neq 0$, nontrivial $2$-cocycles can arise, as the following example shows. 
}
\end{remark}

\begin{example}
    {\rm
    Acomputer aided calculation for the second cohomology group of $\mathfrak{sl}_2(\mathbb C)$ of Definition~\ref{def:Sim_Co} shows that this is $8$-dimensional, i.e. $H^2_{\rm \Sigma SD}(X;X) = \mathbb C^{\oplus 8}$. This example, in particular shows that when simultaneous deformations are considered, SD structures that do not admit deformations (with fixed comultiplication) can be deformed. Interestingly, the counit admits nontrivial deformations as well, showing the the coalgebra structure can be deformed in its entirety, not just through its comultiplication. 
    }
\end{example}

\section{Deformations of Yang-Baxter operators}\label{sec:YB_deform}

We consider now the main application of the cohomology theories introduced and studied in this article. It was shown in \cite{EZ} that ternary SD objects in symmetric monoidal categories can be used to construct braided monoidal caategories. In particular, it was shown that from a ternary SD object $(X,T)$ one can obtain a Yang-Baxter operator on the tensor product $X^{\otimes 2} \otimes X^{\otimes 2}$. The same paradigm is applied to obtain a Yang-Baxter operator on $X^{\otimes 2(n-1)}$, starting from an $n$-ary SD object. In \cite{AZ} it was given a construction that associates an $n$-ary SD object, in the category of vector spaces, and a Yang-Baxter operator starting to an $n$-Lie algebra. Since simultaneous deformations of a given binary (resp. ternary) SD object provide a new $(X',q')$ (resp. $(X',T')$) which is a deformation of $(X,q)$ (resp. $(X,T)$), it follows that the construction of \cite{AZ} gives a corresponding deformation of the associated Yang-Baxter operator. Before investigating this correspondence, we pose the following.

\begin{definition}
	{\rm 
Let $\mathfrak g$ denote an $n$-Lie algebra. Then, let $X$ denote the associated $n$-ary SD object. In this situation, the Yang-Baxter operator corresponding to $X$ is denoted by $R_{\mathfrak g}$. 
}
\end{definition}

As before, we will consider explicitly binary and ternary cases, although the ternary case can be extended to $n$-ary in a straightforward manner. However, our assumption simplifies notation significantly. 

\begin{definition}[\cite{Eis,Eis2}]
	{\rm 
Let $R : V^{\otimes 2} \longrightarrow V^{\otimes 2}$ be a Yang-Baxter operator over the vector space $V$. An operator $\tilde R: V^{\otimes 2} \longrightarrow V^{\otimes 2}$ is said to be a {\it Yang-Baxter deformation} of 
$R$,
or YB deformation for short, is a Yang-Baxter operator such that $[R - \tilde R] (V^{\otimes 2}) \subset \hbar V^{\otimes 2}$. 
}
\end{definition}

We formalize the previous discussion regarding deformations of Yang-Baxter operators arising from (binary) Lie algebras in the following result. 

\begin{lemma}\label{lem:YBdeformations}
	Let $\mathfrak g$ be a binary Lie algebra and let $(\phi,\psi, \iota)$ denote a simultaneous deformation of the associated SD object $(X,\Delta)$. Then, $(\phi,\psi, \iota)$ induces a Yang-Baxter deformation of the operator $R_{\mathfrak g}$ given by 
	$$
	\tilde R_{\mathfrak g}  = R_{\mathfrak g} + \hbar[ (\mathbb 1\otimes \phi)\shuffle (\mathbb 1\otimes \Delta) +  (\mathbb 1\otimes q)\shuffle (\mathbb 1\otimes \psi) ],
	$$
	where $\shuffle = (\mathbb 1\otimes \tau \otimes \mathbb 1)$. 
\end{lemma}
\begin{proof}
	The deformation of $R_{\mathfrak g}$ is the Yang-Baxter operator $\tilde R$ corresponding to the SD object $(X',\Delta')$ (see Theorem~3.4 and Theorem~3.11 in \cite{AZ}), where $X'$ is endowed with the SD map $q' = q + \hbar \phi$, and the coalgebra structure of $X'$ is given by $\Delta' = \Delta + \hbar \psi$, $\epsilon' = \epsilon + \hbar \iota$. To see this explicitly, we write the YB operator $\tilde R_{\mathfrak g} = (\mathbb 1\otimes q')\shuffle (\mathbb 1\otimes \Delta')$ and then, discarding terms with higher powers in $\hbar$ than $1$, we obtain the operator in the statement. 
	 The fact that $R - \tilde R$ maps the space $X'^{\otimes 2}$ into 
	 $\hbar X'^{\otimes 2}$ i.e. $\tilde R$ is a YB deformation, follows readily as well, since $\tilde R_{\mathfrak g}$ differs from $R_{\mathfrak g}$ by a term proportional to $\hbar$. 
\end{proof}

The same constructions hold when one considers a ternary Lie algebra (cf. results in Section~4 of \cite{AZ}). The only difference is that starting with a ternary bracket the construction is ``doubled'', and the corresopnding YB operators are given on the tensor products of the TSD objects. Therefore, an analogous version of Proposition~\ref{lem:YBdeformations} for ternary Lie algebras is obtained simply by changing binary with ternary, and YB operator over $X'$ with YB operator over $X'^{\otimes 2}$. The proof is substantially identical. 

\begin{lemma}\label{lem:terYBdeformations}
	Let $\mathfrak g$ be a ternary Lie algebra and let $(\phi,\psi, \iota)$ denote a simultaneous deformation of the associated SD object $(X,\Delta)$. Then, $(\phi,\psi, \iota)$ induces a Yang-Baxter deformation of the operator $R_{\mathfrak g}$ given by
	$$
	\tilde R_{\mathfrak g}  = R_{\mathfrak g} + \hbar[ (\mathbb 1^{\otimes 2}\otimes \phi)\shuffle (\mathbb 1^{\otimes 2}\otimes \Delta^{\otimes 2}) +  (\mathbb 1^{\otimes 2}\otimes T)\shuffle (\mathbb 1^{\otimes 2}\otimes \psi^{\otimes 2})],
	$$
	where $\shuffle$ is the linearization of the permuation associated to the TSD property.  
\end{lemma}

Let us now restrict our attention to the case of deformations of the SD structure where the coalgebra is fixed. These are simply simultaneous deformations where $\psi$ and $\iota$ are trivial. According to the computations given in Section~\ref{sec:examples}, only three Lie algebras over the reals, and two over the complex, admit deformations where $\psi$ can be nontrivial. In what follows, we use a notation to indicate YB cohomology that is similar to Eisermann's in \cite{Eis,Eis2}, where the YB operator and an ideal $\mathfrak m$ appear explicitly. Since we are dealing with the case where the ring is $\mathbb k[[\hbar]]$ taken along its unique maximal idea $(\hbar)$, we write $H^2_{\rm YB}(R,\hbar)$, or $H^2(R,\hbar)$ for short. We want to show now that SD cohomology classes give rise to YB cohomology classes. Specifically, we have the following.

\begin{theorem}\label{thm:2YB_deform}
	Let notation be as above. Then there is a homomorphism $\Xi^2$ from SD cohomology to Yang-Baxter cohomology $\Xi^2 : H^2_{\rm SD}(X;X) \longrightarrow H^2_{\rm YB}(R,\hbar)$. If, moreover, $\mathfrak g$ has trivial center, $\Xi^2$ is a monomorphism. 
\end{theorem}
\begin{proof}
	Throughout this proof, we use the same notation as  in Lemma~\ref{lem:special}, where maps with image in $X$ have been written as consisting of a sum of a term with image in $\mathbb k$, and one with image in $\mathfrak g$. For $\phi\in C^2_{\rm SD}(X;X)$, we define the following map $c_\phi$, by providing its image on simple tensors as
	$$
	(a,x)\otimes (b,y)\mapsto (b,y)^{(1)}\otimes \phi((a,x)\otimes(b,y)^{(2)}).
	$$
	
First, we want to show that $\Xi^2$ is well defined in cohomology, which accounts to the first part of the statement of the theorem. From Lemma~\ref{lem:YBdeformations} (with $\psi, \iota = 0$) it follows that $\tilde R = R + \hbar c_\phi$ is a YB operator, provided that $\phi$ is an SD $2$-cocycle. But we can write $\tilde R = R\circ (\mathbb 1^{\otimes 2} + \hbar \Xi^2(\phi))$. Applying Proposition~2.9 in \cite{Eis}, it follows that $\Xi^2(\phi)$ is a YB $2$-cocycle.
 To show that $\Xi^2$ passes to the quotient, let us define the map $\Xi^1 : C^1_{\rm SD}(X;X) \longrightarrow C^1_{\rm YB}(R,\hbar)$ by setting $\Xi^1(f) = - f$ for $f\in C^1_{\rm SD}(X;X)$ such that $f(1,0) = 0$, and $\Xi^1(f) = -f + \mathbb 1$ for $f$ such that $f(1,0) \neq 0$. Then, we need to verify that the diagram 
 \begin{center}
 \begin{tikzcd}
 	C^1_{\rm SD}(X;X)\arrow[rr,"\delta^1_{\rm SD}"]\arrow[d,swap,"\Xi^1"] & & C^2_{\rm SD}(X;X)\arrow[d,"\Xi^2"]\\
 	C^1_{\rm YB}(R;\hbar)\arrow[rr,swap,"\delta^1_{\rm YB}"] & & C^2_{\rm YB}(R;\hbar)
 	\end{tikzcd}
 \end{center}
 is commutative. This is substantially the condition for $\Xi^1,\Xi^2$ to define a cochain map up to degree $2$. A direct computation on simple tensors shows that equality holds.
 
 Let us now consider injectivity of $\Xi^2$, under the extra assumption that $\mathfrak g$ has trivial center. We need to show that if $\Xi^2(\phi)$ is cobounded, then $\phi$ is conounded as well. In other words, we have to show that if $R^{-1}c_\phi = \delta^1_{\rm YB}g$ for some $g$, there exists a $f : X\longrightarrow X$ such that $\phi = \delta^1_{\rm SD}f$. 	Since by Lemma~\ref{lem:special} $\phi$ is special, the map $c_\phi$ becomes $c_\phi((a,x)\otimes (b,y)) = (1,0)\otimes \phi((0,x)\otimes (0,y))$. Then, we define $\Xi^2(\phi) =  R^{-1}c_\phi$.
 The equality $R^{-1}c_\phi = \delta^1_{\rm YB}g$ is equivalent to showing that $c_\phi = R\delta^1_{\rm YB}g$, which is simpler to deal with.
  For the RHS, we have 
 \begin{eqnarray*}
 R\delta^1_{\rm YB} f((a,x)\otimes (b,y)) &=& f(b,y)^{(1)}\otimes q((a,x)\otimes g(b,y)^{(2)}) + (b,y)^{(1)}\otimes q(g(a,x)\otimes (b,y)^{(2)})\\
 && - [(\mathbb 1\otimes g + g\otimes \mathbb 1)](b,y)^{(1)}\otimes q((a,x)\otimes (b,y)^{(2)})\\
 &=& (1,0)\otimes (ag_0(y),g_0(y)x + [x,g_1(y)]) + (1,0)\otimes (0,[g_1(a,x),y])\\
 && - (1,0)\otimes g(0,[x,y]) - g(1,0)\otimes (0,[x,y]).
 \end{eqnarray*}
Now, suppose that $g_0$ is not the trivial map, i.e. the projection of $g(0,y)$ over $\mathbb k$ is nontrivial for some $y$. Then, the RHS of equation $c_\phi = R\delta^{1}_{\rm YB}f$ would have a nontrivial term $(1,0)\otimes (1,0)$ for some $y$. But this is not possible since the LHS has no such term due to the fact that $\phi$ is special (Lemma~\ref{lem:special}). Then, it follows that $g_0$ is the zero map. We have found that 
\begin{eqnarray*}
 R\delta^1_{\rm YB} f((a,x)\otimes (b,y))  &=&   (1,0)\otimes (0, [x,g_1(y)]) + (1,0)\otimes (0,[g_1(a,x),y])\\
 && - (1,0)\otimes g(0,[x,y]) - g(1,0)\otimes (0,[x,y]),
\end{eqnarray*}
and moreover, $g(1,0) = g_0(1,0)$ since there are no terms that have the first tensorand different from $(1,0)$. Observe that if we had that $g(1,0) = 0$ we could simply choose $g = -f$ and 
it would follow that $\phi = \delta^1_{\rm SD}f$. Suppose otherwise. So $g(1,0) = k$ for some $k\in \mathbb k$. Without loss of generality we can assume that $k = 1$, since we could otherwise divide $g$ by $k$, and apply the previous steps with not further modifications. In this case, we can define $f = - g + \mathbb 1$, and a direct computation shows that the equality $\phi = \delta^1_{\rm SD}f$ holds in this case as well, completing the proof. 
\end{proof}

Similarly, one can prove the same result for $n$-Lie algebras. We provide the statement of the analogous result for $n=3$, proof of which follows the same lines given above.

\begin{theorem}\label{thm:nYB_deform}
	Let notation be as above. Then there is a homomorphism $\Xi^2$ from SD cohomology to Yang-Baxter cohomology $\Xi^2 : H^2_{\rm SD}(X;X) \longrightarrow H^2_{\rm YB}(R,\hbar)$. If, moreover, $\mathfrak g$ has trivial center, $\Xi^2$ is a monomorphism. 
\end{theorem}

In particular, we observe that when $\mathfrak g$ has trivial center (as an $n$-ary Lie algebra) we see that the Lie cohomology is a direct summand of the Yang-Baxter cohomology. In other words, nontrivial Lie deformations produce nontrivial YB deformations of the associated YB operator.

\begin{corollary}
    Let $\mathfrak g$ be an $n$-Lie algebra with trivial center. Then the Lie algebra cohomology injects into the YB cohomology of the operator $R_{\mathfrak g}$. 
\end{corollary}
\begin{proof}
    This is a direct consequence of Theorem~\ref{thm:TSDdeform} and Theorems~\ref{thm:2YB_deform} and \ref{thm:nYB_deform}.
\end{proof}

We conclude this section with a few comments on possible topological applications of this theory. Part of these comments were also mentioned in \cite{AZ}, which constitutes the first step of the overall program. Some of the remarks presented in \cite{AZ} were in fact addressed in this article. The fundamental objective is to derive link invariants from the YB operators associated to given $n$-ary Lie algebras. Here, we can think of the $n$-ary version of the approach as a ``colored'' instance of the binary one. The cohomology theory considered here allows to deform the SD structure associated to an $n$-ary Lie algebra, and therefore to produce deformed YB operators whose traces (which are in practice used for invariants \cite{Tur_YB,Oht}) are polynomials with respect to the variable $\hbar$. In the case that the Lie algebra can be deformed arbitrarily many times, one might expects that such deformations propagate to the corresponding SD structure and YB operator. Therefore the corresponding invariants would be power series on the formal parameter $\hbar$. Interestingly, for the case of YB operators infinitesimally deformed, one can show that the two coefficients (of degree zero and one in $\hbar$) are both link invariants, therefore suggesting that in the higher deformation case each coefficient of the power series might be a link invariant as well. 

   \section*{Acknowledgement} 
Mohamed Elhamdadi was partially supported by Simons Foundation collaboration grant 712462.
	
	\newpage
	\appendix
	\section{Proof of Lemma~\ref{lem:differentials}}
	
    In this appendix we show that the differentials defined in Section~\ref{sec:TSD_coh} compose to zero. 
	The proof is by direct computation on basis monomials of $X^{\otimes 5}$. We need to show that $\delta^1f$ is in $C^2_{\rm TSD}(X;X)$, and that $\delta^2\delta^1f = 0$ for every coderivation $f$. The first statement accounts to showing that the following diagram commutes
	\begin{center}
		\begin{tikzcd}
			X^{\otimes 3}\arrow[rr,"\shuffle\circ\Delta_3^{\otimes 3}"]\arrow[d,swap,"\delta^1f"] && X^{\otimes 9}\arrow[d,"\delta^1f\otimes T^{\otimes 2}+T\otimes \delta^1f\otimes T+T^{\otimes 2}\otimes \delta^1f"] \\
			X\arrow[rr,swap,"\Delta_3"] && X^{\otimes 3}
			\end{tikzcd}
	\end{center}
	 Let us consider the lower perimeter first. For $x\otimes y\otimes z\in X^{\otimes 3}$ we have 
	\begin{eqnarray*}
	\Delta_3\delta^1f(x\otimes y\otimes z) &=& \Delta_3(f(T(x\otimes y\otimes z) - T(f(x)\otimes y\otimes z))\\
	&& - T(x\otimes f(y)\otimes z) - T(x\otimes y\otimes f(z)))\\
	&=& (f\otimes \mathbb 1\otimes \mathbb 1)(T^{\otimes 3}(x^{(1)}\otimes y^{(1)}\otimes z^{(1)}\otimes x^{(2)}\otimes y^{(2)}\otimes z^{(2)}\otimes x^{(3)}\otimes y^{(3)}\otimes z^{(3)}))\\
	&& + ( \mathbb 1\otimes f\otimes \mathbb 1)(T^{\otimes 3}(x^{(1)}\otimes y^{(1)}\otimes z^{(1)}\otimes x^{(2)}\otimes y^{(2)}\otimes z^{(2)}\otimes x^{(3)}\otimes y^{(3)}\otimes z^{(3)}))\\
	&&+ (\mathbb 1\otimes \mathbb 1\otimes f)(T^{\otimes 3}(x^{(1)}\otimes y^{(1)}\otimes z^{(1)}\otimes x^{(2)}\otimes y^{(2)}\otimes z^{(2)}\otimes x^{(3)}\otimes y^{(3)}\otimes z^{(3)}))\\
	&& - T^{\otimes 3}(f(x^{(1)})\otimes y^{(1)}\otimes z^{(1)}\otimes x^{(2)}\otimes y^{(2)}\otimes z^{(2)}\otimes x^{(3)}\otimes y^{(3)}\otimes z^{(3)})\\
	&& -T^{\otimes 3}(x^{(1)}\otimes y^{(1)}\otimes z^{(1)}\otimes f(x^{(2)})\otimes y^{(2)}\otimes z^{(2)}\otimes x^{(3)}\otimes y^{(3)}\otimes z^{(3)})\\
	&& - T^{\otimes 3}(x^{(1)}\otimes y^{(1)}\otimes z^{(1)}\otimes x^{(2)}\otimes y^{(2)}\otimes z^{(2)}\otimes f(x^{(3)})\otimes y^{(3)}\otimes z^{(3)})\\
	&& - T^{\otimes 3}x^{(1)}\otimes f(y^{(1)})\otimes z^{(1)}\otimes x^{(2)}\otimes y^{(2)}\otimes z^{(2)}\otimes x^{(3)}\otimes y^{(3)}\otimes z^{(3)})\\
	&& -T^{\otimes 3}(x^{(1)}\otimes y^{(1)}\otimes z^{(1)}\otimes x^{(2)}\otimes f(y^{(2)})\otimes z^{(2)}\otimes x^{(3)}\otimes y^{(3)}\otimes z^{(3)})\\
	&& - T^{\otimes 3}(x^{(1)}\otimes y^{(1)}\otimes z^{(1)}\otimes x^{(2)}\otimes y^{(2)}\otimes z^{(2)}\otimes x^{(3)}\otimes f(y^{(3)})\otimes z^{(3)})\\
	&& - T^{\otimes 3}(x^{(1)}\otimes y^{(1)}\otimes f(z^{(1)})\otimes x^{(2)}\otimes y^{(2)}\otimes z^{(2)}\otimes x^{(3)}\otimes y^{(3)}\otimes z^{(3)})\\
	&& -T^{\otimes 3}(x^{(1)}\otimes y^{(1)}\otimes z^{(1)}\otimes x^{(2)}\otimes y^{(2)}\otimes f(z^{(2)})\otimes x^{(3)}\otimes y^{(3)}\otimes z^{(3)})\\
	&& - T^{\otimes 3}(x^{(1)}\otimes y^{(1)}\otimes z^{(1)}\otimes x^{(2)}\otimes y^{(2)}\otimes z^{(2)}\otimes x^{(3)}\otimes y^{(3)}\otimes f(z^{(3)})),
	\end{eqnarray*}
	where we have used the fact that $f$ is a coderivation to rewrite the term $ \Delta_3(f(T(x\otimes y\otimes z)) - T(f(x)\otimes y\otimes z)$, and the fact that $T$ is a coalgebra morphism (by definition of TSD object) and the coderivation property again, to rewrite the terms $- \Delta_3T(f(x)\otimes y\otimes z)- \Delta_3T(x\otimes f(y)\otimes z) - \Delta_3T(x\otimes y\otimes f(z)))$. This is seen to coincide with the upper perimeter by direct computation. 
	
	For the second statement, i.e. the fact that composition of differentials is trivial on coderivations $f: X\longrightarrow X$, we evaluate $\delta^2(\delta^1f)$ on simple tensors $x_1\otimes x_2\otimes x_3\otimes x_4 \otimes x_5$. We have
	\begin{eqnarray*}
	\lefteqn{\delta^2(\delta^1f)(x_1\otimes x_2\otimes x_3\otimes x_4\otimes x_5)}\\
	&=& T(f(T(x_1\otimes x_2\otimes x_3))\otimes x_4\otimes x_5)- T(T(f(x_1)\otimes x_2\otimes x_3)\otimes x_4\otimes x_5)\\
	&& - T(T(x_1\otimes f(x_2)\otimes x_3)\otimes x_4\otimes x_5) - T(T(x_1\otimes x_2\otimes f(x_3))\otimes x_4\otimes x_5)\\
	&& + f(T(T(x_1\otimes x_2\otimes x_3)\otimes x_4\otimes x_5)) - T(f(T(x_1\otimes x_2\otimes x_3))\otimes x_4\otimes x_5)\\
	&& - T(T(x_1\otimes x_2\otimes x_3)\otimes f(x_4)\otimes x_5) - T(T(x_1\otimes x_2\otimes x_3)\otimes x_4\otimes f(x_5))\\
	&& - f(T(T(x_1\otimes x_4^{(1)}\otimes x_5^{(1)})\otimes T(x_2\otimes x_4^{(2)}\otimes x_5^{2})\otimes T(x_3\otimes x_4^{(3)}\otimes x_5^{3})))\\
	&& +  T(f(T(x_1\otimes x_4^{(1)}\otimes x_5^{(1)}))\otimes T(x_2\otimes x_4^{(2)}\otimes x_5^{2})\otimes T(x_3\otimes x_4^{(3)}\otimes x_5^{3}))\\
	&& + T(T(x_1\otimes x_4^{(1)}\otimes x_5^{(1)})\otimes f(T(x_2\otimes x_4^{(2)}\otimes x_5^{2}))\otimes T(x_3\otimes x_4^{(3)}\otimes x_5^{3}))\\
	&& + T(T(x_1\otimes x_4^{(1)}\otimes x_5^{(1)})\otimes T(x_2\otimes x_4^{(2)}\otimes x_5^{2})\otimes f(T(x_3\otimes x_4^{(3)}\otimes x_5^{3})))\\
	&& - T(f(T(x_1\otimes x_4^{(1)}\otimes x_5^{(1)}))\otimes T(x_2\otimes x_4^{(2)}\otimes x_5^{2})\otimes T(x_3\otimes x_4^{(3)}\otimes x_5^{3}))\\
	&& + T(T(f(x_1)\otimes x_4^{(1)}\otimes x_5^{(1)})\otimes T(x_2\otimes x_4^{(2)}\otimes x_5^{2})\otimes T(x_3\otimes x_4^{(3)}\otimes x_5^{3}))\\
	&& +  T(T(x_1\otimes f(x_4^{(1)})\otimes x_5^{(1)})\otimes T(x_2\otimes x_4^{(2)}\otimes x_5^{2})\otimes T(x_3\otimes x_4^{(3)}\otimes x_5^{3}))\\
	&& + T(T(x_1\otimes x_4^{(1)}\otimes f(x_5^{(1)}))\otimes T(x_2\otimes x_4^{(2)}\otimes x_5^{2})\otimes T(x_3\otimes x_4^{(3)}\otimes x_5^{3}))\\
	&& - T(T(x_1\otimes x_4^{(1)}\otimes x_5^{(1)})\otimes f(T(x_2\otimes x_4^{(2)}\otimes x_5^{2}))\otimes T(x_3\otimes x_4^{(3)}\otimes x_5^{3}))\\
	&& + T(T(x_1\otimes x_4^{(1)}\otimes x_5^{(1)})\otimes T(f(x_2)\otimes x_4^{(2)}\otimes x_5^{2})\otimes T(x_3\otimes x_4^{(3)}\otimes x_5^{3}))\\
	&& +  T(T(x_1\otimes x_4^{(1)}\otimes x_5^{(1)})\otimes T(x_2\otimes f(x_4^{(2)})\otimes x_5^{2})\otimes T(x_3\otimes x_4^{(3)}\otimes x_5^{3}))\\
	&& + T(T(x_1\otimes x_4^{(1)}\otimes x_5^{(1)}\otimes T(x_2\otimes x_4^{(2)}\otimes f(x_5^{2}))\otimes T(x_3\otimes x_4^{(3)}\otimes x_5^{3}))\\
	&& - T(T(x_1\otimes x_4^{(1)}\otimes x_5^{(1)})\otimes T(x_2\otimes x_4^{(2)}\otimes x_5^{2})\otimes f(T(x_3\otimes x_4^{(3)}\otimes x_5^{3})))\\
	&& + T(T(x_1\otimes x_4^{(1)}\otimes x_5^{(1)})\otimes T(x_2\otimes x_4^{(2)}\otimes x_5^{2})\otimes T(f(x_3)\otimes x_4^{(3)}\otimes x_5^{3}))\\
	&& +  T(T(x_1\otimes x_4^{(1)}\otimes x_5^{(1)})\otimes T(x_2\otimes x_4^{(2)}\otimes x_5^{2})\otimes T(x_3\otimes f(x_4^{(3)})\otimes x_5^{3}))\\
	&& + T(T(x_1\otimes x_4^{(1)}\otimes x_5^{(1)})\otimes T(x_2\otimes x_4^{(2)}\otimes x_5^{2})\otimes T(x_3\otimes x_4^{(3)}\otimes f(x_5^{3}))).
	\end{eqnarray*}
    It is a direct verification to see that all terms cancel in pairs, with exception of terms that differ by an application of the coderivation condition of $f$ and therefore cancel out anyways. The composition vanishes identically for every simple tensor of $X^{\otimes 5}$, therefore completing the proof.

\end{document}